\documentclass{article}
\usepackage{amssymb,amsmath,graphicx,hyperref}
\usepackage[noadjust]{cite}
\usepackage[utf8x]{inputenc}
\usepackage[english]{babel}
\usepackage[ruled,vlined,norelsize]{algorithm2e}

\newtheorem{theorem}{Theorem}
\newtheorem{proposition}[theorem]{Proposition}
\newtheorem{corollary}[theorem]{Corollary}
\newtheorem{lemma}[theorem]{Lemma}

\newenvironment{proof}{\noindent \emph{Proof. }}{\hfill \hbox{\rlap{$\sqcap$}$\sqcup$}\\}

\title{Weak Colored Local Rules for Planar Tilings
\thanks{This work was supported by the ANR project QuasiCool (ANR-12-JS02-011-01)}}

\author{
Thomas Fernique
\footnote{Univ. Paris 13, CNRS, Sorbonne Paris Cit\'e, UMR 7030, 93430 Villetaneuse, France.}
\and
Mathieu Sablik
\footnote{Institut de Math\'ematiques de Toulouse, UMR 5219, Univ. de Toulouse, CNRS, Univ. P. Sabatier 
}
}
\date{}

\begin{document}

\maketitle

\begin{abstract}
  A linear subspace $E$ of $\mathbb{R}^n$ has {\em colored local rules} if there exists a finite set of decorated tiles whose tilings are digitizations of $E$. 
  The local rules are {\em weak} if the digitizations can slightly wander around $E$.
  We prove that a linear subspace has weak colored local rules if and only if it is computable.
  This goes beyond the previous results, all based on algebraic subspaces.
  We prove an analogous characterization for sets of linear subspaces, including the set of all the linear subspaces of $\mathbb{R}^n$.
\end{abstract}

\section{Introduction}

A {\em tiling} of a given space is a covering by interior-disjoint compacts called {\em tiles}.
The shapes of tiles yield constraints on the way they can fit together, as do the bumps and dents in a jigsaw puzzle.
The way these local constraints affect the global structure of tilings in non-trivial.
In particular, it was proven in the 1960's that there is no algorithm which, given a finite tile set as input, answers in finite time whether one can form a tiling of the whole plane with these tiles (each tile can be translated and used several times): this is known as the {\em undecidability of the domino problem} (\cite{Wang-1960} and \cite{Berger-1966}, or \cite{Robinson-1971} for a self-contained exposition).
The first key ingredients of the proof is the simulation of Turing computations by tilings of the plane.
The second one is the existence of {\em aperiodic tile sets}, that are finite tile sets which tile the plane but only in a non-periodic fashion.\\

The interest in aperiodic tile sets received a boost two decades later when new non-periodic crystals (soon called {\em quasicrystals}) were incidentally discovered by the chemist Dan Shechtman \cite{Shechtman-1984}.
The connection with tilings was indeed quickly made, with tiles modelling atom clusters and the constraints on the way they can fit together modelling finite range energetic interactions \cite{Levine-Steinhardt-1984}.
A classification of all the possible quasicrystalline structures, in the spirit of the Bravais-Fedorov classification of crystalline structures, is still to be found.\\

A popular way to obtain quasicrystalline structure is the {\em cut and projection} method, introduced by De Bruijn for the celebrated Penrose tilings in \cite{deBruijn-1981}.
The principle is to approximate an affine $d$-plane of $\mathbb{R}^n$, obtaining a tilings by rhomboedra, called {\em planar tiling}.
Rhomboedral tiles are of course specific, but they already provide a very rich framework to deal with (let us mention that Kenyon showed in \cite{Kenyon-1992} that any tiling problem whose tiles are topological closed ball of the plane can be reduced to a problem with polygonal tiles, so that rhomboedral tiles may be not that specific).
In particular, the rhomboedral tiles of a given planar tilings can also form tilings approximating other planes, or even tilings that are not planar at all.
The challenge is thus to enforce tiles to approximate the wanted planes by specifying rules on the ways they can locally fit together -- one speaks about {\em local rules}.\\

A systematic approach, initiated by Levitov in \cite{Levitov-1988} and further developped by Le and Socolar in \cite{Socolar-1990,Le-1995,Le-1997}, aims to characterize all the irrational planes whose approximating tilings can be defined by such local rules.
In particular, Levitov provided a powerful sufficient condition (the so-called {\em SI-condition}) and Le showed that such planes can always be defined by algebraic irrationalities, that is, there is an {\em algrebraic obstruction}.
However, no complete characterization has yet be found.\\

Local rules can also be naturally enriched by coloring tiles, each in a color from a given finite palette.
This indeed goes back to the first works on aperiodic tile set mentioned at the beginning of the introduction, where tiles have colored edges and can match only along identically colored edges -- one speaks about {\em matching rules} or {\em colored local rules}.
Rephrased in terms of symbolic dynamics following \cite{Robinson-2004}, this means considering {\em sofic multi-dimensional subshifts} instead of only {\em finite type} ones.
Many planes have been shown to be approximated by tilings which can be defined by colored local rules ({\em e.g.}, \cite{deBruijn-1981,Socolar-1989,Le-Piunikhin-Sadov-1992,Le-Piunikhin-1995,Le-1995,Fernique-Ollinger-2010}).
But a complete characterization is here also missing.
In particular, one can wonder whether the previous algebraic obstruction of Le still holds (all the mentioned planes are indeed algebraic).\\

In this paper, we show that the algebraic obstruction breaks down when colors are allowed and must be replaced by a {\em computability obstruction}, which moreover turns out to be tight: an irrational plane admits colored local rules if and only if it can be defined by computable irrationalities.
We thus get a complete characterization, indeed the first one.
This can be compared to the landmark result of Hochman and Meyerovich \cite{Hochman-Meyerovich-2010}, who characterized the values of entropies of multidimensional shifts of finite type (the tilings with colored local rules being a generalization of) by a computation-theoretic property (namely right recursively enumerability).
Computability thus seems to play a natural role in multidimensional symbolic dynamics.\\

Actually, our main theorem (Theorem~\ref{th:main}) deals not only with computable planes but with computable {\em sets} of planes.
The case of a single plane is an immediate corollary, but a surprising corollary is that colored local rules can also enforce the set of all planes, that is, planarity itself.
We also show that colors can even be removed by encoding them into slight fluctuations around the approximated planes; this does not contradict the algebraic obstruction of Le because, as we shall see, the arising local rules are different from thoses introduced by Levitov.\\

We follow a computational approach, as in the proof of the undecidability of the domino problem.
We use colored local rules to encode simulations of Turing computations which check that only planar tilings that approximate the wanted planes can be formed.
The fundamental tool is a recent result in symbolic dynamics, which states that {\em any effective one-dimensional subshift can be obtained as the subaction of a two-dimensional sofic subshift} \cite{Aubrun-Sablik-2013,Durand-Romashchenko-Shen-2012} and also \cite{Hochman-2009}.
Besides this, our main ingredient is a slight extension of Sturmian words, called {\em quasi\-sturmian words}.
Roughly, quasisturmian words allow to split the checking of the parameters of approximated planes into a product of independant checking of each parameter.\\

The paper is organized as follows.
In Section~\ref{sec:settings}, we introduce the formalism and state our main result, Theorem~\ref{th:main}.
We also provide an extensive survey of existing results concerning local rules for cut and projection tilings.
Section~\ref{sec:computability_obstruction} shows that one cannot go beyond computability with colored local rules.
The converse is proven in Section~\ref{sec:weak_colored_rules} after introducing quasisturmian words in Section~\ref{sec:quasisturmian}.
Section~\ref{sec:uncolor} finally shows that colored local rules can actually be simulated by uncolored ones at the price of a slightly coarser plane digitization (we coined the term {\em weakened} local rules to distinguish it from weak local rules).

%%%%%%%%%%%%%%%%%%%%%%%%%%%%%%%%%%%%%%%%%
%%%%%%%%%%%%%%%%%%%%%%%%%%%%%%%%%%%%%%%%%
%%%%%%%%%%%%%%%%%%%%%%%%%%%%%%%%%%%%%%%%%
\section{Settings}
\label{sec:settings}

\subsection{Planar tilings}

Let us first recall some basics definitions on tilings, following \cite{Robinson-2004,Sadun-2008,Baake-Grimm-2013}.
A {\em tile} is a compact subset of $\mathbb{R}^d$ which is the closure of its interior.
Two tiles are equivalent if one is a translation of the other; the equivalence classes representatives are called {\em prototiles}.
A {\em tiling} of $\mathbb{R}^d$ is a covering by interior-disjoint tiles.
Given two tilings $\mathcal{T}$ and $\mathcal{T}'$, let $R(\mathcal{T},\mathcal{T}')$ be the supremum of all radii $r$ such that $\mathcal{T}$ and $\mathcal{T}'$ can be translated by a vector shorter than $\frac{1}{2r}$ in order to agree on a ball of radius $r$ around the origin.
The {\em distance} between two tilings is then defined as the smaller of $1$ and $1/R(\mathcal{T},\mathcal{T}')$.
In other words, two tilings are close if, up to a small translation, they agree on a large ball around the origin.
This distance defines a topology which makes the set of all tilings compact, provided that it has {\em finite local complexity}, that is, there is only finitely many ways two tiles can be adjacent in a tiling.
A set of tilings which is closed and invariant by translation is called a {\em tiling space}.\\

Let us define the specific tilings that we shall consider throughout this paper.
They have been introduced in theoretical physics as random tiling models (see, {\em e.g.}, \cite{Henley-1991}).
Let $\vec{v}_1,\ldots,\vec{v}_n$ be pairwise non-colinear vectors of $\mathbb{R}^d$, $n>d>0$.
A {\em $n\to d$ prototile} is a parallelotope generated by $d$ of the $\vec{v}_i$'s, {\em i.e.}, the linear combinations with coefficient in $[0,1]$ of $d$ of the $\vec{v}_i$'s.
Then, a {\em $n\to d$ tiling} is a {\em face-to-face} tiling of $\mathbb{R}^d$ by $n\to d$ tiles, {\em i.e.}, a covering of $\mathbb{R}^d$ by $n\to d$ tiles which can intersect only on full faces of dimension less than $d$.
Face-to-face condition ensures finite local complexity, hence compactness of the set of $n\to d$ tilings.\\

We shall now explain how to {\em lift} a $n\to d$ tiling (following \cite{Levitov-1988}).
Let $\vec{e}_1,\ldots,\vec{e}_n$ be the canonical basis of $\mathbb{R}^n$.
Given a $n\to d$ tiling, we first map an arbitrary vertex onto the origin, then we map each tile generated by $\vec{v}_{i_1},\ldots,\vec{v}_{i_d}$ onto the $d$-dimensional face of a unit hypercube of $\mathbb{Z}^n$ generated by $\vec{e}_{i_1},\ldots,\vec{e}_{i_d}$, with two tiles adjacent along the edge $\vec{v}_i$ being mapped onto two faces adjacent along the edge $\vec{e}_i$.
This defines, up to the choice of the initial vertex, the {\em lift} of the tiling.
This is a digital $d$-dimensional manifold in $\mathbb{R}^n$, and $d$ and $n-d$ are respectively called the dimension and the codimension of the tiling.\\

We are now in a position to define planarity.
A $n\to d$ tiling is {\em planar} if there is an affine $d$-dimensional plane (a $d$-plane) $E\subset \mathbb{R}^n$ and $t\geq 1$ such that this tiling can be lifted into the tube $E+[0,t]^n$.
The space $E$ is called the {\em slope} of the tiling and the smallest suitable $t$ -- its {\em thickness}.
One checks that both the slope and the thickness of a planar tiling are uniquely defined.
The case when $t=1$ and $E$ contains no integer point brings us back to the classical cut and projection tilings \cite{deBruijn-1986,Gahler-Rhyner-1986}.
For larger $t$, fluctuations are allowed which do not affect the long range order of the lift.
By extension, a tiling space is said to be planar if its tilings are all planar, and it is said to be uniformly planar if the thickness of its tilings is uniformly bounded.\\

\subsection{Local rules}
\label{sec:local_rules}

A finite set of tiles occuring in a tiling is called a {\em patch}.
One calls {\em local rules} a finite set of patches.
A set $\mathcal{F}$ of local rules defines the set $X_{\mathcal{F}}$ of the tilings in which no patch of $\mathcal{F}$ does occur -- the patches in $\mathcal{F}$ are said to be {\em forbidden}.
The set $X_{\mathcal{F}}$ turns out to be a tiling space; following the symbolic dynamics terminology of \cite{Robinson-2004}, it is said to have {\em finite type}.\\

As mentioned in the introduction, local rules can be {\em colored}: each tile can be endowed with one color from a given finite palette.
This defines colored tilings and, by removing colors, a tiling space.
Following again the symbolic dynamics terminology of \cite{Robinson-2004}, such a tiling space is said to be {\em sofic}.
Colored local rules are actually equivalent to so-called {\em matching rules}, where tiles are decorated ({\em e.g.} by coloring edges) and allowed to fit together only if their decorations match, as do the Wang tiles introduced in \cite{Wang-1960} or the two arrowed rhombi discovered by Penrose \cite{Penrose-1974}.
Any finite type tiling space is sofic but the converse does not always hold.\\

Let us now focus on planar tilings.
In \cite{Levitov-1988}, Levitov introduced the notions of {\em strong local rules} and {\em weak local rules}, that we here complete by {\em weakened local rules}.
Formally, a set $\mathcal{E}$ of $d$-planes of $\mathbb{R}^n$ is said to {\em admit} (or to {\em be enforced by}) {\em weakened local rules} if there is a set $\mathcal{F}$ of local rules and an integer $t\geq 1$, called the {\em thickness} of theses local rules, such that
\begin{itemize}
\item $X_{\mathcal{F}}$ contains a planar tiling of thickness at most $t$ with a slope parallel to $E$ for each $E\in\mathcal{E}$;
\item $X_{\mathcal{F}}$ contains only planar tilings of thickness $t$ with a slope parallel to some element of $\mathcal{E}$.
\end{itemize}
Weakened local rules become {\em weak} (resp. {\em strong}) local rules if we moreover assume $t=1$ in the first (resp. both) of the two above conditions.
Note that $X_{\mathcal{F}}$ does not necessarily contain all the planar tilings of thickness $t$ with a slope parallel to an element of $\mathcal{E}$: this is why we prefer to say that local rules enforce planes, not tilings.
All this naturally extends to colored local rules.

\subsection{Computability}
\label{sec:computability}

A complete introduction to computable analysis can be found in \cite{Weihrauch-2000}.
We here just briefly recall the notion that will be used in this paper.\\

First, a real number is said to be {\em computable} if it can be approximated by a rational within any desired precision by a finite, terminating algorithm.
For example, $\pi=3.1415\ldots$ can be approximated by the partial summations of numerous series, but not the real number whose $i$-th binary digit is $0$ if and only if the $i$-th computer program ({\em e.g.}, following the lexicographic order) halts.\\

More generally, a metric space $(X,d)$ is said to be {\em computable} if it contains a countable dense subset, called the {\em rational points} of $X$, and a computable function $\rho$ which gives the distance between any two rational points of $X$.
An open subset $A\subset X$ is then said to be {\em recursively open} if there is a non-terminating algorithm that outputs infinite sequences $(q_n)$ of rational points and $(r_n)$ of rational numbers such that $A$ is the union of the open balls of center $q_n$ and radius $r_n$ (one speaks about an {\em effective enumeration} of $A$).
The complement of a recursively open set is said be {\em recursively closed}.
When the space is compact, one shows that $x\in X$ is computable iff the set $\{x\}$ is recursively closed.\\

Throughout this paper, we shall be particularly interested in two specific metrical spaces.
The first one is the set of real numbers endowed with the Euclidean distance, with the rational points being the usual rational numbers.
The second one is the space of the $d$-planes of $\mathbb{R}^n$ endowed with the metric
$$
d(E,F)=\max\left\{\sup_{\vec{x}\in E\cap S} \inf\{||\vec{x}-\vec{y}||~:~\vec{y}\in F\},\sup_{\vec{x}\in F\cap S} \inf\{||\vec{x}-\vec{y}||~:~\vec{y}\in E\}\right\},
$$
where $S$ denotes the unit sphere of $\mathbb{R}^n$.
This space is known to be compact. %\marginpar{référence ?}
Its rational points are the $d$-planes generated by vectors with rational entries.%\marginpar{distance calculable ?}

%\todo{exemples : tout, un plan, une courbe algébrique comme les plans de Beenker}

\subsection{Results}
\label{sec:results}

We already provided in the introduction an overview of the existing results about local rules for planar tilings.
Let us here classify them according to the dual distinction weak/strong and colored/uncolored, trying to be as comprehensive as possible.
We shall then state the main result of this paper, which is a characterization for colored weak local rules.

\paragraph{Colored strong local rules.}
They can be traced back to \cite{deBruijn-1981}, when de Bruijn proved that the Penrose tilings by rhombi - introduced with colored local rules in \cite{Penrose-1978} - are digitizations of an plane in $\mathbb{R}^5$ based on the golden ratio ({\em i.e.}, generated by vectors with entries in $\mathbb{Q}[\sqrt{5}]$).
Beenker \cite{Beenker-1982} then tried to find colored strong local rules for another example, namely a plane in $\mathbb{R}^4$ based on the silver ratio (generated by vectors with entries in $\mathbb{Q}[\sqrt{2}]$), unaware that it was already known by the ``Mysterious Mr. Ammann'' \cite{Senechal-2008} in a work that was only later published in \cite{Ammann-Grunbaum-Shephard-1992}.
This case appears also in \cite{Socolar-1989}, where another example in $\mathbb{R}^6$ is also provided (based on the same quadratic field).
The Penrose case was extended in \cite{Le-1995} for some parallel planes.
Let us also mention \cite{Mozes-1989,Goodman-Strauss-1998,Fernique-Ollinger-2010}, where the tilings obtained by so-called substitutions are proved to admit co\-lo\-red local rules; this apply to some planar tilings, but only with algebraic slopes.

\paragraph{Uncolored strong local rules.}
They appeared more or less as the same time as colored one.
The Penrose tilings are indeed known to also admit such local rules (see, {e.g.}, \cite{Senechal-1995}, p.177).
The already mentioned work of Beenker \cite{Beenker-1982} deserves to be here recalled, since looking for uncolored local rules for a particular plane, he actually found a one-parameter family of planes that was later used to prove that the wanted local rules do not exist \cite{Burkov-1988}.
It was moreover shown by Katz \cite{Katz-1995} that this whole family admits colored strong local rules, and uncolored one if the parameter ranges through a rational interval (see also \cite{Bedaride-Fernique-2013}).
This shows that local rules can not only characterize single planes but also {\em sets} of planes.
Last but not least, in \cite{Levitov-1988}, Levitov provided a necessary condition on a two- or three-dimensional linear subspace of $\mathbb{R}^n$ to admit uncolored strong local rules (a sort of rational dependency between the entries of vectors generating the slope).

\paragraph{Uncolored weak local rules.}
The adjectives ``strong'' and ``weak'' for local rules have actually been coined in \cite{Levitov-1988}.
There, Levitov proved that its necessary condition for uncolored strong local rules is also a sufficent one for uncolored weak local rules, at least for linear subspaces of $\mathbb{R}^4$ and some other particular cases, namely the generalized Penrose tilings introduced in \cite{Pavlovitch-Kleman-1987} and the icosahedral tilings (whose slope is a three-dimensional subspace of $\mathbb{R}^6$ based again on the golden ratio), see also \cite{Katz-1988,Socolar-1990} for this latter.
In \cite{Socolar-1990}, it is proven that the planar tiling with a $n$-fold rotational symmetry admits (rather simple) uncolored weak local rules as soon as $n$ is not a multiple of $4$.
It has then been proven that there is no such rules as soon as $n$ is a multiple of $4$ (first for $n=8$ in \cite{Burkov-1988}, then for any $n$ in \cite{Bedaride-Fernique-2015b}).
A complete characterization for two-dimensional linear subspace of $\mathbb{R}^4$ finally emerged in \cite{Bedaride-Fernique-2015,Bedaride-Fernique-2016}; in particular only quadratic slopes can be characterized by uncolored weak local rules.
The characterization for $d$-dimensional linear subspace of $\mathbb{R}^n$ is still to be obtained.
Maybe one of the most remarkable result in this way is the obstruction proved by Le in \cite{Le-1997}, who showed that such linear spaces are necessarily defined by vectors of an algebraic number field.

\paragraph{Colored weak local rules.}
Although computational considerations were central in the first colored local rules, see {\em e.g.}, \cite{Wang-1960,Berger-1966}, this was not the case in the (rare) known results for colored weak local rules.
They have indeed been shown in \cite{Le-Piunikhin-Sadov-1992} to exist for any quadratic planes of $\mathbb{R}^k$, a result then extended to quadratic $d$-dim. planes of $\mathbb{R}^{2d}$ in \cite{Le-Piunikhin-1995}.
This paper brings back computational consideration in this framework and provides a complete characterization of all the linear subspace, and even sets of such subspaces, that admit colored weak local rules.
We shall indeed prove that colored weak local rules allow to go much further than the algebraic obstruction of \cite{Le-1997} for uncolored weak local rules.
Formally:

\begin{theorem}\label{th:main}
A set of planes is enforced by colored weak local rules if and only if it is recursively closed.
\end{theorem}

We moreover show that colored weak local rules can be replaced in the above theorem by (uncolored) weakened local rules for any set of planes satisfying a condition of non-degeneration (see Corollary~\ref{cor:uncolor}, Section~\ref{sec:uncolor}).
Let us also state two immediate corollaries:

\begin{corollary}
A plane is enforced by colored weak local rules if and only if it is computable.
\end{corollary}

\begin{corollary}\label{cor:planarity_is_sofic}
The set of all planes is enforced by colored weak local rules.
\end{corollary}

%%%%%%%%%%%%%%%%%%%%%%%%%%%%%%%%%%%%%%%%%
%%%%%%%%%%%%%%%%%%%%%%%%%%%%%%%%%%%%%%%%%
%%%%%%%%%%%%%%%%%%%%%%%%%%%%%%%%%%%%%%%%%
\section{Computability obstruction}
\label{sec:computability_obstruction}

\subsection{The case of a single plane}

\begin{proposition}\label{prop:une_pente}
A plane which admits local rules (of any type) is recursive.
\end{proposition}

\begin{proof}
Let $E$ be a plane admitting local rules of thickness $t$.
We associate with a patch $P$ the set $s(P)$ of the slopes of the thickness $t$ planar tilings $P$ is a patch of:
$$
s(P):=\{E \textrm{ affine subspace of } \mathbb{R}^n~|~P\subset E+[0,t]^n\}.
$$
Computing such a set is a classical problem in discrete geometry (see, {\em e.g.}, \cite{Gerard-Debled-Zimmermann-2005}).
Actually, the classical case is the one of a $2$-plane $E$ in $\mathbb{R}^3$ (because of the applications in digital imagery), but the methods easily extend to higher dimension and codimension.\\

We call {\em $r$-patch} a patch whose tiles cover the ball of radius $r$ centered at the origin and all intersect it.
If $P$ is an $r$-patch of a tiling satisfying the local rules, then any slope of $s(P)$ is at distance at most $t/r$ from $E$, whence $s(P)$ has diameter at most $2t/r$.
However, it is unclear how to compute such a patch (think about the {\em deceptions} of \cite{Dworkin-Shieh-1995}).\\

To deal with this, we introduce, for any $r'\geq r$, the set $\mathcal{P}_{r,r'}$ of the restrictions of the $r'$-patches satisfying the local rules to their tiles which intersect the ball of radius $r$ centered at the origin.
This is a set of $r$-patches which can be computed (though in exponential time).
By compacity of the tiling space defined by the local rules, for a large enough $r'$, $\mathcal{P}_{r,r'}$ contains only $r$-patches of a tiling satisfying the local rules and the diameter of $s(\mathcal{P}_{r,r'})$ is thus less than $2t/r$.\\

We can now give our algorithm.
We increment $r'$ and compute $s(\mathcal{P}_{r,r'})$ until it reaches a diameter less than $2t/r$.
Then, we compute a $r$-patch $P'\in\mathcal{P}_{r,r'}$.
On the one hand, whenever $P$ is a $r$-patch of a tiling satisfying the local rules, any slope in $s(P)$ yields an approximation within $t/r$ of $E$.
On the other hand, two slopes in $s(P')$ and $s(P)$ are at distance bounded by $2t/r$, the diameter of $s(\mathcal{P}_{r,r'})$.
Any slope in $s(P')$ thus yields an approximation within $3t/r$ of $E$.
Choosing a suitable $r$ yields the desired precision.
\end{proof}

This can be summarized by the following algorithm:

\begin{algorithm}
\KwData{Local rules of thickness $t$ enforcing a plane $E$, an integer $m$}
\KwResult{Approximation within $1/m$ of $E$}
$r\gets 3tm$\;
$r'\gets r$\;
\Repeat{$d\leq\frac{2}{3m}$}{
$r'\gets r'+1$\;
$d\gets $ diameter of $s(\mathcal{P}_{r,r'})$\;}
\KwRet an element of $s(\mathcal{P}_{r,r'})$\;
\label{algo:une_pente}
\caption{Approximation of the slope of a plane with local rules}
\end{algorithm}

Algorithm~\ref{algo:une_pente} can be adapted when the thickness is not given: different copies of the algorithm run in parallel with growing ``guessed'' thickness $t=1,2,\ldots$ until one of them indeed halts and returns the desired approximation.
The first algorithm to halt is not necessarily the one with the true thickness.
Actually, one cannot compute the thickness of local rules enforcing a plane, nor decide whether local rules enforce a plane.

\subsection{General case}

Let us show that the previous result extends to the case of a set of planes by suitably modifying the algorithm.

\begin{proposition}\label{prop:ens_pentes}
A set of planes wich admits local rules is recursively closed.
\end{proposition}

\begin{proof}
Let $\mathcal{E}$ be a set of planes which admits local rules (of any type) of thickness $t$.
Let $(F_k)_{k\in\mathbb{N}}$ be a recursive enumeration of all the rational $d$-planes of $\mathbb{R}^n$.
We shall show that Algorithm~\ref{algo:ens_pente} (below) enumerates closed balls whose union is the complement of $\mathcal{E}$.
We keep the notation $s(\mathcal{P}_{r,r'})$ introduced in the proof of Proposition~\ref{prop:une_pente}.\\

First, consider a slope $E\in\mathcal{E}$.
For any $r'\geq r$, there exists $E'\in s(\mathcal{P}_{r,r'})$ such that $d(E,E')<\frac{t}{r}$.
If Algorithm~\ref{algo:ens_pente} outputs a ball $B(F_k,\frac{t}{r})$, then $d(F_k,E')\geq \frac{2t}{r}$.
Thus $d(F_k,E)\geq \frac{t}{r}$ and $E$ is never in a ball enumerated by Algorithm~\ref{algo:ens_pente}.\\

Now, consider a slope $G\notin\mathcal{E}$.
Let $\varepsilon=d(G,\mathcal{E})$.
It is positive since $\mathcal{E}$ is closed.
Fix $r\in\mathbb{N}$ such that $\varepsilon>\frac{3t}{r}$.
By density of the rational planes, there is $k\in\mathbb{N}$ such that $G\in B(F_k,\frac{t}{r})$.
We shall show that $B(F_k,\frac{t}{r})$ is enumerated by Algorithm~\ref{algo:ens_pente}.\\

First, the choices of $r$ and $F_k$ ensures that $d(F_k,\mathcal{E})> \frac{2t}{r}$.
Then, by compacity of the tiling space defined by the local rules, for a large enough $r'$, $\mathcal{P}_{r,r'}$ contains only $r$-patches of a tiling satisfying the local rules.
There is thus $E\in\mathcal{E}$ such that $d(E',E)\leq \frac{t}{r}$ for any $E'\in s(\mathcal{P}_{r,r'})$.
Fix such an $E'$.
The triangle inequality yields $d(F_k,E)\leq d(F_k,E')+d(E',E)$, that is, $d(F_k,E')\geq d(F_k,E)-d(E',E)$.
With $d(F_k,E)> \frac{2t}{r}$ and $d(E',E)\leq \frac{t}{r}$ this ensures that $d(F_k,E')>\frac{t}{r}$.
The ball $B(F_k,\frac{t}{r})$ is thus eventually enumerated by Algorithm~\ref{algo:ens_pente}.
\end{proof}

\begin{algorithm}
\KwData{Local rules of thickness $t$ enforcing a set of planes $\mathcal{E}$}
\KwResult{Enumeration of closed balls whose union is the complement of $\mathcal{E}$}
$m\gets 1$\;
$r'\gets 1$\;
$t\gets 1$\;
\While{True}{
\For{$r=1$ \KwTo $r=r'$}{
\For{$k=1$ \KwTo $k=m$}{
\If{$\min(d(F_k,E'),~E'\in s(\mathcal{P}_{r,r'}))\geq\frac{t}{r}$}{
\KwOut{$B(F_k,\frac{t}{r})$}
}
}
}
$m\gets m+1$\;
$r'\gets r'+1$\;
$t\gets t+1$\;
}
\label{algo:ens_pente}
\caption{Enumeration of the slopes of a set of planes with local rules}
\end{algorithm}

%%%%%%%%%%%%%%%%%%%%%%%%%%%%%%%%%%%%%%%%%
%%%%%%%%%%%%%%%%%%%%%%%%%%%%%%%%%%%%%%%%%
%%%%%%%%%%%%%%%%%%%%%%%%%%%%%%%%%%%%%%%%%
\section{Quasisturmian subshifts}
\label{sec:quasisturmian}

\subsection{Effective subshifts as sofic subactions}

The issue of local rules, introduced above in the context of tilings, also appears in the context of symbolic dynamics.
In this context, the role of tiles is played by letters over some finite alphabet $\mathcal{A}$, the role of tilings -- by multidimensional words indexed by $\mathbb{Z}^n$ called {\em configurations}, and the role of a tiling space -- by a set of configurations called {\em subshift} avoiding some set of ``forbidden'' finite patterns (a pattern is a word indexed by some finite subset of $\mathbb{Z}^n$).
A subshift is said to be of {\em finite type} if it can be defined by finitely many forbidden patterns.
It is said to be {\em sofic} if it is the image of a finite type subshift (over a possibly larger alphabet) under a letter-to-letter map called {\em factor map} (which plays the role of the color removal for tilings).
Last, a subshift is said to be {\em effective} or {\em recursive} if there exists a Turing machine which enumerates all the patterns that do not appear in any of its configurations.\\

One of the major recent results that binds computability and symbolic dynamics is the following one, proven indeendantly in \cite{Aubrun-Sablik-2013} and \cite{Durand-Romashchenko-Shen-2012}, both impro\-ving \cite{Hochman-2009}.
We shall strongly rely on it to prove Theorem~\ref{th:main}.

\begin{theorem}[\cite{Aubrun-Sablik-2013,Durand-Romashchenko-Shen-2012}]
\label{th:subaction}
If $S\subset\mathcal{A}^{\mathbb{Z}^n}$ is an effective subshift, then the following subshift is sofic:
$$
\widetilde{S}=\{x\in\mathcal{A}^{\mathbb{Z}^{n+1}} : \textrm{there is } y\in S \textrm{ such that }x(m,\cdotp)=y \textrm{ for all }m\in\mathbb{Z} \},
$$
where $x(m,\cdotp)\in\mathcal{A}^{\mathbb{Z}^n}$ is the $m$-th ``row'' of $x$, {\i.e.}, its first entry is set to $m$.
\end{theorem}

The subshift $S$ is called a {\em subaction} of the subshift $\widetilde{S}$: the latter is indeed made of parallel copies of configurations of the former.
In other words, this theorem states that any effective $n$-dimensional subshift can be realized as a subaction of a $n+1$-dimensional sofic subshift.
In order to get a better picture of the construction developed in the following, it may be worthwhile to give an idea of the proof of the above result (following \cite{Aubrun-Sablik-2013}):
\begin{itemize}
\item the sofic subshift simulates a Turing machine;
\item this Turing machine runs along the extra-dimension of $\widetilde{S}$ (time dimension);
\item parallel copies of a given configuration $x\in S$ make it readable at anytime;
\item the Turing machine enumerates the forbidden patterns;
\item it checks that no enumerated forbidden pattern appears in $x$.
\end{itemize}
Of course, such a construction requires a huge number of tiles (which depends on the Kolmogorov complexity of the forbidden patterns).
Moreover, the time to detect a forbidden pattern can be very large, so that assembling a tiling tile by tile is highly unrealistic: errors can usually be detected only after many tiles have been added.
This construction however perfectly suits to prove the theoretical results of this paper.

\subsection{From tilings to subshifts}

Consider the lift of a $3\to 2$ tiling.
Let us orthogonally project along $\vec{e}_1+\vec{e}_2$.
A tile which contains an edge directed by $\vec{e}_3$ is projected onto a square, say of type $1$ or $2$ depending on whether the second edge of the tile is directed by $\vec{e}_1$ or $\vec{e}_2$.
A tile by $\vec{e}_1$ and $\vec{e}_2$ is projected onto a segment that we choose to ignore.
We thus get a tiling of the plane of normal vector $\vec{e}_1+\vec{e}_2$ by two types of squares\footnote{
  Except if the lift of the tiling is bounded in the direction $\vec{e}_3$, that is, if half of the tiling is made of tiles with edges directed by $\vec{e}_1$ and $\vec{e}_2$.
  In the case of a planar tiling, it means that the slope has normal vector $\vec{e}_3$.
  We consider this as a degenerated case (see also Sec.~\ref{sec:uncolor}) and we ignore it in what follows, since it is trivial to find local rules for such a planar tiling.}.
It can naturally be seen as a configuration in $\{1,2\}^{\mathbb{Z}^2}$, with $\vec{e}_3$ giving the vertical.
Figure~\ref{fig:from_3vers2_to_2vers2} illustrates this.\\

\begin{figure}[hbtp]
\includegraphics[width=\textwidth]{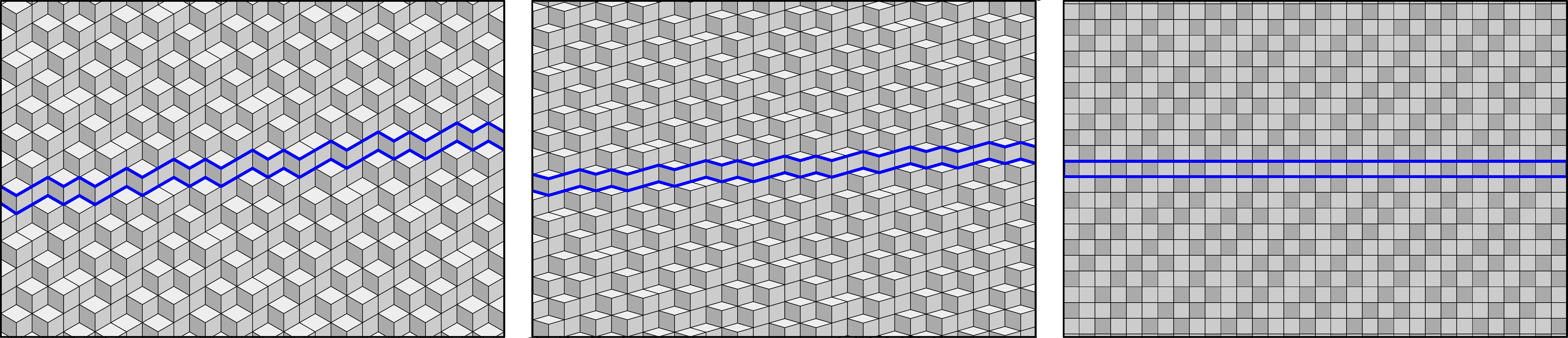}
\caption{Projection (with an intermediary step) of a $3\to 2$ tiling onto a configuration in $\{1,2\}^{\mathbb{Z}^2}$.}
\label{fig:from_3vers2_to_2vers2}
\end{figure}

By permuting edges, one can actually associate three configurations with any $3\to 2$ tiling.
Then, if $\mathcal{S}$ denote a set of $2$-planes in $\mathbb{R}^3$, let $S^i_{\mathcal{S}}$, $i=1,2,3$, denote the set of configurations associated with the $3\to 2$ tilings with slope in $\mathcal{S}$: these are $2$-dimensional subshifts on two letters.
We conjecture that when $\mathcal{S}$ is recursively closed, then these subshifts are sofic, but we have yet no proof.
Instead, we shall introduce slightly larger subshifts (namely {\em quasisturmian shifts}) and use Theorem~\ref{th:subaction} to prove that they are sofic.
The corresponding tiling space is also slightly larger: it contains planar tilings with the desired slopes but a possibly larger thickness (whence we get weak colored local rules and not strong ones).

\subsection{Sturmian subshifts}

A line of a configuration associated as above with a $3\to 2$ tiling corresponds to a ``stripe'' (sometimes referred to as {\em de Bruijn line} or {\em worm}), as depicted on Figure~\ref{fig:from_3vers2_to_2vers2} (framed tiles).
If we consider a planar $3\to 2$ tiling of thickness $1$ whose slope has normal vector $(\alpha,\beta,\gamma)$, then each of its stripe is contained into some tube $\vec{u}+\mathbb{R}(\alpha,\beta,0)+[0,1]^3$, where $\vec{u}\in\mathbb{Z}^3$.
The corresponding line in the associated configuration thus turns out to be a {\em Sturmian word} of slope $|\alpha/\beta|$.
Recall (see, {\em e.g.}, \cite{Lothaire-2002}) that the Sturmian word $s_{\alpha,\rho}\in\{1,2\}^{\mathbb{Z}}$ of {\em slope} $\alpha\in[0,1]$ and {\em intercept} $\rho\in[0,1]$ can be defined by
$$
s_{\alpha,\rho}(n)=1
~\Leftrightarrow~
(\rho+n\alpha)\mod 1\in[0,1-\alpha).
  $$
  By exchanging letters $0$ and $1$ in the Sturmian word $s_{\alpha,\rho}$ one gets then gets the Sturmian word of slope $1/\alpha$, so that any slope can be achieved.
A configuration associated with a planar $3\to 2$ tiling of thickness $1$ is thus made of stacked Sturmian words, all with the same slope but with different intercepts (depending on the parameter $\gamma$).
In order to use Theorem~\ref{th:subaction}, we would like to avoid these intercepts variations.
We therefore introduce {\em Sturmian subshifts}: those are the $2$-dim. subshifts whose configurations are formed by stacked Sturmian words, all with the same slope and the same intercept.
Formally, we associate with any non-empty subset $A\subset\mathbb{R}^+$ the Sturmian subshift
\begin{equation}\label{eq:S_A}
S_A=\{x\in\{0,1\}^{\mathbb{Z}^2}~:~\exists \alpha\in A,~\exists \rho\in[0,1],~\forall m\in\mathbb{Z},~x(m,\cdotp)=s_{\alpha,\rho}\},
\end{equation}
where $x(m,\cdotp)$ denotes the $m$-th row of $x$.
Theorem~\ref{th:subaction} and the proposition below then ensure that $S_A$ is sofic for a recursively closed $A$.

\begin{proposition}\label{prop:sofic_sturmians_computable_slopes}
The set of Sturmian words with a slope in a given recursively closed subset of $\mathbb{R}$ form an effective subshift.
\end{proposition}

\begin{proof}
Let us first give an algorithm to eventually detect if a finite word $u$ is forbidden (the algorithm halts iff $u$ is forbidden).
We compute the slopes of the (infinite) Sturmian words $u$ can be a factor of.
This is a (possibly empty) rational open interval $I(u)$ which can be computed in almost linear time by classic methods (see \cite{Klette-Rosenfeld-2003}).
Let $(B_n)_n$ be an enumeration of rational open balls whose union is the complement of our recursively closed set of slopes.
The word $u$ is forbidden if and only if the union of the $k$-th first balls eventually  (when $k$ grows) contains $I(u)$.\\
Let us now use this algorithm to enumerate all the forbidden words.
We simply browse all the finite words ({\em e.g.}, by lexicographic order) and run the algorithm on each of them ``in parallel'', that is, we run one step of the algorithm on the $k$-th first browsed words before browsing the $k+1$-th one.
\end{proof}

\subsection{Relaxation}

The Sturmian subshift $S_A$ is sofic when $A$ is recursively closed, but unfortunately the subshift derived from a $3\to 2$ planar tiling are not exactly of this type (as already mentioned).
We shall here define a subshift $S'_A$ which is still sofic when $A$ is recursively closed and which contains the subshift derived from a set of $3\to 2$ planar tilings with suitable slopes.\\

Let us first define quasisturmian words.
Consider the set $\{0,1\}^{\mathbb{Z}}$ of bi-infinite words over the alphabet $\{0,1\}$ endowed with the metric $d$ defined by
$$
d(u,v):=\sup_{p\leq q}\left||u(p)u(p+1)\ldots u(q)|_0-|v(p)v(p+1)\ldots v(q)|_0\right|,
$$
where $w(k)$ denotes the $k$-th letter of $w$ and $|.|_0$ counts the occurences of $0$.
In other terms, the distance between two words is the maximum {\em balance} between their finite factors which begin and start at the same positions.
The {\em quasisturmian words} of slope $\alpha$ are the words in $\{0,1\}^{\mathbb{Z}}$ at distance at most one from a Sturmian word of slope $\alpha$.\\

We now can define {\em quasisturmian subshifts}: those are the $2$-dim. subshifts whose configurations are formed by stacked quasiturmian words, all with the same slope.
Formally, we associate with any non-empty closed subset $A\subset\mathbb{R}^+$ the quasisturmian subshift
\begin{equation}\label{eq:S'_A}
S'_A=\{x\in\{0,1\}^{\mathbb{Z}^2}~:~\exists \alpha\in A,~\exists \rho\in[0,1],~\forall m\in\mathbb{Z},~d(x(m,\cdotp),s_{\alpha,\rho})\leq 1\},
\end{equation}
where $x(m,\cdotp)$ denotes the $m$-th row of $x$.
Clearly $S_A\subset S'_A$, and the following proposition ensures that $S'_A$ also contains the subshift derived from a set of $3\to 2$ planar tilings with the set $\mathcal{S}$ of slopes such that $A=\{|\alpha/\beta|~:~\exists \gamma,~(\alpha,\beta,\gamma)\in \mathcal{S}\}$ ({\em w.l.o.g.}, we consider the projection along $\vec{e}_1+\vec{e}_2$).

\begin{proposition}\label{prop:parallel_sturmian}
Sturmian words with equal slopes are at distance at most one.
\end{proposition}

\begin{proof}
Two sturmian words $u$ and $v$ with equal slopes are known to have the same finite factors.
Any two factors of respectively $u$ and $v$ which begin and start at the same positions are thus also factors of $u$ only - at different position but with the same number of letters.
This yields the bound
$$
d(u,v)\leq\sup_{p,q,r}\left||u(p)u(p+1)\ldots u(p+r)|_0-|v(q)v(q+1)\ldots v(q+r)|_0\right|.
$$
This bound is known to be at most one for Sturmian words (and only them).
\end{proof}

Let us now show that $S'_A$ is sofic when $A$ is recursively closed.
We shall use the following lemma:

\begin{lemma}\label{lem:alternated_replacements}
Two words in $\{0,1\}^{\mathbb{Z}}$ are at distance at most one if and only if each can be obtained from the other by performing letter replacements $0\to 1$ or $1\to 0$, without two consecutive replacements of the same type.
\end{lemma}

\begin{proof}
Let $u$ and $v$ in $\{0,1\}^{\mathbb{Z}}$ at distance at most one.
Performing on $u$ replacements at each position $i$ where $u(i)\neq v(i)$ yields $v$.
If two consecutive replacements, say at position $p$ and $q$, have the same type, then the balance between $u(p)\ldots u(q)$ and $v(p)\ldots v(q)$ is two, hence $d(u,v)\geq 2$.
The type of replacements thus necessarily alternates.\\
Conversely, assume that $v\in\{0,1\}^{\mathbb{Z}}$ is obtained from $u\in\{0,1\}^{\mathbb{Z}}$ by per\-for\-ming replacements whose type alternates.
Given $p\leq q$, consider the number of replacements between positions $p$ and $q$: the balance between $u(p)\ldots u(q)$ and $v(p)\ldots v(q)$ is $0$ if this number is even, $1$ otherwise, hence $d(u,v)\leq 1$.
\end{proof}

Since the replacements to transform $u$ in $v$ alternate, their sequence can be encoded by $w\in\{0,1\}^{\mathbb{Z}}$: reading $01$ (resp. $10$) at position $i$ means that a replacement $0\to 1$ (resp $1\to 0$) occurs at position $i$.
Such a word $w$ is moreover unique, except if $u=v$ in which case both $w=0^{\mathbb{Z}}$ and $w=1^{\mathbb{Z}}$ suit.
Figure \ref{fig:quasisturmian} illustrates this.
We use this coding in the next proposition.\\

\begin{figure}[hbtp]
  \centering
  \includegraphics[width=\textwidth]{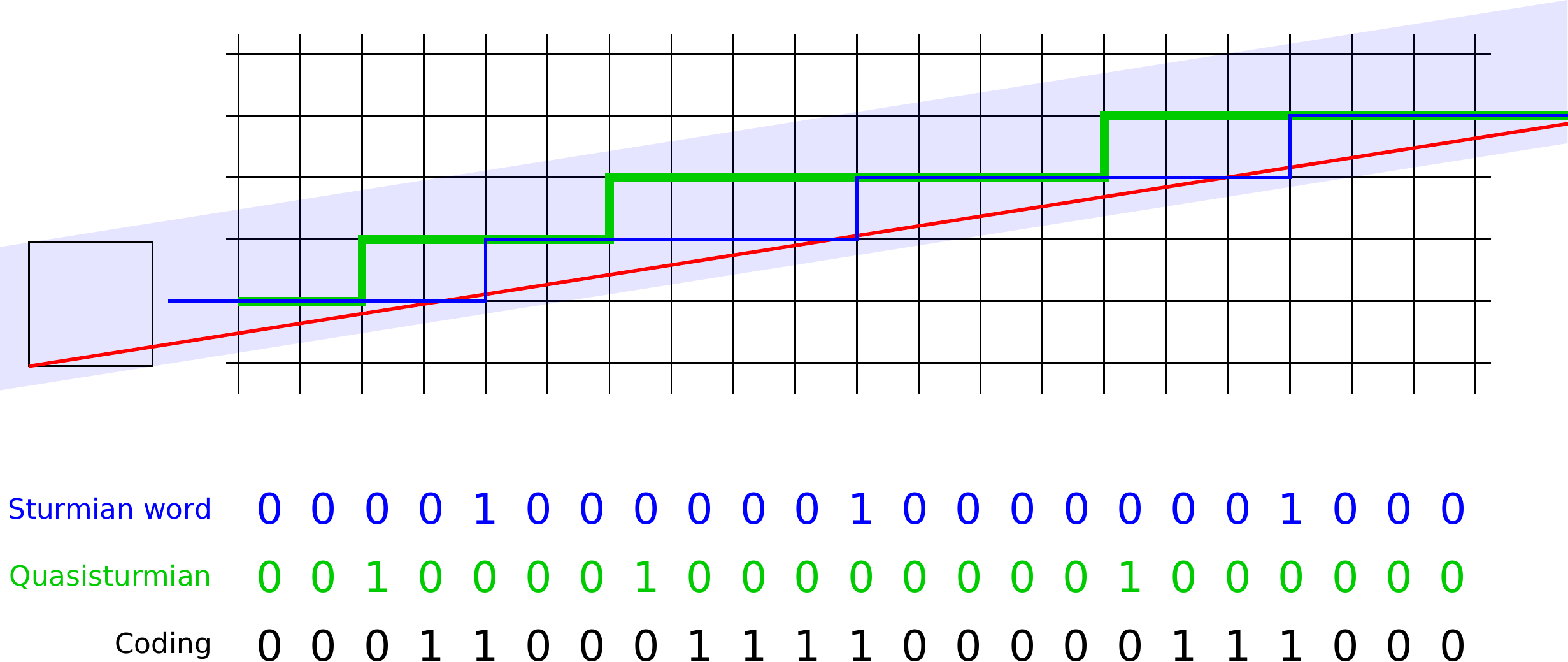}
  \caption{A Sturmian word, a quasisturmian word with the same slope, their codings and the coding of the transformation from the former to the latter.}
\label{fig:quasisturmian}
\end{figure}

\begin{proposition}\label{prop:recursive_subshift}
If $A$ is a recursively closed set, then $S'_A$ is a sofic subshift.
\end{proposition}

\begin{proof}
  Let $\pi_{i_1,\ldots,i_k}$ denotes the map which selects the entries $i_1,\ldots,i_k$ of a tuple.  
Since $S_A$ is sofic, there is an alphabet $B$ and a $2$-dim. subshift of finite type $\tilde{S}_A$ over $(\{0,1\} \times B)$ such that
$$
S_A=\pi_1(\tilde{S}_A).
$$
We shall now prove that $S'_A$ is sofic by defining a $2$-dim. subshift of finite type $\tilde{S}'_A$ and a factor map $\pi$ from $\tilde{S}'_A$ onto $S'_A$.
The idea is to add to the configurations of $\tilde{S}_A$ a third entry that will encode (through Lemma~\ref{lem:alternated_replacements}) the difference between the Sturmian words on their first entry and the quasisturmian words on the rows of $S'_A$.
Formally, let $\tilde{S}'_A$ be the $2$-dim. subshift over $(\{0,1\}\times B\times \{0,1\})$ such that $u\in\tilde{S}'_A$ if and only if $\pi_{1,2}(u)\in \tilde{S}_A$ and, for any $(m,n)\in\mathbb{Z}^2$:
$$
\pi_3(u(m,n))<\pi_3(u(m,n+1)) ~\Rightarrow~ \pi_1(u(m,n))=0,
$$
$$
\pi_3(u(m,n))>\pi_3(u(m,n+1)) ~\Rightarrow~ \pi_1(u(m,n))=1.
$$
The subshift $\tilde{S}'_A$ is of finite type because so does $\tilde{S}_A$ and the third entry in a given position of a configuration only depends on the neighboor positions.
Now, let $\pi$ be the factor map defined on $\tilde{S}'_A$ by
$$
\pi(u)(m,n)=\left\{\begin{array}{cl}
\pi_1(u(m,n)) & \textrm{if }\pi_3(u(m,n))=\pi_3(u(m,n+1)),\\
1-\pi_1(u(m,n)) & \textrm{otherwise.}
\end{array}\right.
$$
First, let us show that $\pi(\tilde{S}'_A)\subset S'_A$.
Let $\tilde{u}\in\tilde{S}'_A$ and fix $m\in\mathbb{Z}$.
By definition of $\tilde{S}'_A$ and $\tilde{S}_A$, $\pi_1(\tilde{u}(m,\cdotp))=s_{\alpha,\rho}$.
One thus also has $\pi(\tilde{u}(m,\cdotp))=s_{\alpha,\rho}$, except at each position $n$ such that the two bits $\pi_3(\tilde{u}(m,n))$ and $\pi_3(\tilde{u}(m,n+1))$ differ.
At these positions, $\pi(\tilde{u}(m,\cdotp))$ is obtained by performing on $s_{\alpha,\rho}$ a replacement of type $\pi_3(\tilde{u}(m,n))\to \pi_3(\tilde{u}(m,n+1))$.
The type of these replacements alternate - as the bit runs do - and Lemma~\ref{lem:alternated_replacements} yields $d(\pi(\tilde{u}(m,\cdotp)),s_{\alpha,\rho})\leq 1$.
This shows that $\pi(\tilde{u})$ is in $S'_A$.
Hence, $\pi(\tilde{S}'_A)\subset S'_A$.\\
Now, let us show that $S'_A\subset\pi(\tilde{S}'_A)$.
Let $u\in S'_A$. Fix $m\in\mathbb{Z}$ and choose $\tilde{v}\in \tilde{S}_A$ such that $\pi_1(\tilde{v}(m,\cdotp))=s_{\alpha,0}$. By definition, $d(u(m,\cdotp),s_{\alpha,0})\leq 1$, so we can consider $w_m$  the coding of the replacements which transform $s_{\alpha,0}$ into $u(m,\cdotp)$.
Consider $\tilde{u}\in (\{0,1\}\times B\times \{0,1\})^{\mathbb{Z}^2}$ defined by $\tilde{u}(m,i)=(\tilde{v}(m,i),w_m(i))$.
The way $\pi$ has been defined yields $\tilde{u}\in\tilde{S}'_A$ and $\pi(\tilde{u})=u$.
Hence, $S'_A\subset \pi(\tilde{S}'_A)$.
\end{proof}

\subsection{Summary}

For the sake of clarity, it can be useful to end this section by briefly summarizing the numerous subshifts introduced in the previous subsections.\\

We initially associated with any set $\mathcal{S}$ of $2$-planes in $\mathbb{R}^3$ the subshifts $S^i_{\mathcal{S}}$, whose configurations are projections of $3\to 2$ tilings with slope in $\mathcal{S}$.
The lines of a configuration in this subshift are Sturmian words with the same slope, but the intercept varies in a non-trivial way from line to line, according to the slope of the projected $3\to 2$ tiling.
We conjectured that such a subshift is sofic when $\mathcal{S}$ is recursively enumerable, but because we have no proof of this we introduced larger subshifts whose soficity can be proved.\\

The first idea was to force the lines of each configuration to have all the same intercept.
This lead to introduce (Eq.~\ref{eq:S_A}):
$$
S_A=\{x\in\{0,1\}^{\mathbb{Z}^2}~:~\exists \alpha\in A,~\exists \rho\in[0,1],~\forall m\in\mathbb{Z},~x(m,\cdotp)=s_{\alpha,\rho}\},
$$
where $A$ is a set of slopes characterized by $\mathcal{S}$.
This subshift was easily proven to be sofic when $\mathcal{S}$ is recursively enumerable : this is a rather straightforward corollary of Theorem~\ref{th:subaction}.
But it does not contain $S^i_\mathcal{S}$: the constraints on the intercept of lines are too strict.\\

We therefore tried to relax the constraints on the intercept.
Namely, we allowed the lines of configurations to be quasisturmian words (Eq.~\ref{eq:S'_A}):
$$
S'_A=\{x\in\{0,1\}^{\mathbb{Z}^2}~:~\exists \alpha\in A,~\exists \rho\in[0,1],~\forall m\in\mathbb{Z},~d(x(m,\cdotp),s_{\alpha,\rho})\leq 1\}.
$$
In order to emphasize that $S'_A$ is a relaxation of $S_A$, one can equivalently write:
$$
S'_A=\{x\in\{0,1\}^{\mathbb{Z}^2}~:~\exists y\in S_A,~\forall m\in\mathbb{Z},~d(x(m,\cdotp),y(m,\cdotp))\leq 1\}.
$$
We proved that this allows the intercept to vary freely (Prop.~\ref{prop:parallel_sturmian}), so that $S'_A$ contains $S^i_\mathcal{S}$ (and, of course, $S_A$).
Moreover, we used that $S_A$ is sofic (when $\mathcal{S}$ is recursively enumerable) to obtain that $S'_A$ is also sofic (Prop.~\ref{prop:recursive_subshift}).
This is the result that shall be used in the next section.\\

Note that $S'_A$ is the relaxation of $S_A$ but not of $S^i_{\mathcal{S}}$ (as pointed out to us by Emmanuel Jeandel).
Note also that it could seem more natural, instead of introducting quasisturmian words, to simply allow the intercept to freely vary on each line, that is, to replace $S_A'$ by the subshift
$$
S''_A=\{x\in\{0,1\}^{\mathbb{Z}^2}~:~\exists \alpha\in A,~\forall m\in\mathbb{Z},~\exists \rho\in[0,1],~x(m,\cdotp)=s_{\alpha,\rho}\}.
$$
However, this latter subshift is (generally) not sofic:

\begin{proposition}\label{prop:sturmian_lines_non_sofic}
  If $A\not\subset \mathbb{Q}$, then the subshift $S''_A$ is not sofic.
\end{proposition}

A proof of this result, which is not used later in the paper, is given in Appendix~\ref{sec:appendix}.

%%%%%%%%%%%%%%%%%%%%%%%%%%%%%%%%%%%%%%%%%
%%%%%%%%%%%%%%%%%%%%%%%%%%%%%%%%%%%%%%%%%
%%%%%%%%%%%%%%%%%%%%%%%%%%%%%%%%%%%%%%%%%
\section{Weak colored local rules}
\label{sec:weak_colored_rules}

\subsection{Any or all of the $2$-planes in $\mathbb{R}^3$}

We shall here prove the following:

\begin{proposition}\label{prop:weak_local_rules_one_plane}
Any computable $2$-plane in $\mathbb{R}^3$ has weak colored local rules of thickness $2$.
\end{proposition}

\begin{proof}
Consider a $2$-plane in $\mathbb{R}^3$ with a computable normal vector $(1,\alpha,\beta)$.
Prop.~\ref{prop:recursive_subshift} ensures that $S'_{\alpha}$ is sofic (for the sake of simplicity, $S'_{\alpha}$ stands for $S'_{\{\alpha\}}$).
Let us see this subshift as the tiling space of a Wang tile set (recall that Wang tiles are square with colored edges that can be adjacent only along full edges with the same color), with the letters $1$ or $2$ being written inside the tiles.\\

Let us shear these square tiles along $\vec{e}_3$ to get tiles of type $1$ or $2$, according to the letter written inside the tiles.
We then add all the tiles of type $3$ needed to transfer along the direction $\vec{e}_3$ any decoration appearing on the $\vec{e}_1$ or $\vec{e}_2$ edges of tiles of type $1$ or $2$.
Fig.~\ref{fig:shearing_wang} illustrates this.
This yields a tile set which exactly forms the $3\to 2$ tilings whose orthogonal projection along $\vec{e}_1+\vec{e}_2$ yields $S'_{\alpha}$.\\

By proceeding in the same way, we can get a second tile set which exactly form the $3\to 2$ tilings whose orthogonal projection along $\vec{e}_1+\vec{e}_3$ yields $S'_{\beta}$.
We now join these two tile sets into a single one by cartesian product: whenever a tile of the first set and one of the second set have the same shape, we define a new tile (with the same shape) and encode in a one-to-one way the colors on the edges of the two original tiles into a color on the edge of the new tile.
This yields a finite tile set which forms $3\to 2$ tilings whose associated subshifts are $S'_{\alpha}$, $S'_{\beta}$ and $S'_{\alpha/\beta}$.\\

It remains to show that these tilings stay at bounded distance from the $2$-plane with normal vector $(1,\alpha,\beta)$.
Let us call {\em $\vec{v}_i$-ribbon} of a $3\to 2$ tiling a maximal sequence of tiles, with two consecutive tiles being adjacent along an edge $\vec{v}_i$ ; such a ribbon is said to be {\em directed} by $\vec{v}_i$.
Consider two vertices $x$ and $y$ of such a tiling.
Each of them is the endpoint of edges which takes at least two different directions.
There are thus $i\neq j$ such that $x$ belongs to a $\vec{v}_i$-directed edge - hence to a $\vec{v}_i$-directed ribbon - and $y$ to a $\vec{v}_j$-directed edge - hence to a $\vec{v}_j$-directed ribbon.
For simplicity, assume $i=2$ and $j=3$.
These two ribbons intersect: let $z$ be a vertex in this intersection.
Fig.~\ref{fig:connecting_ribbons} illustrates this.
The $\vec{v}_2$-directed ribbon (resp. the $\vec{v}_3$-directed one) is a quasisturmian word with slope $\alpha$ (resp. $\beta$).
Their lifts are thus respectively contained in some tubes
$$
X+\mathbb{R}\vec{u}+[0,1]\times[0,1]\times[0,2],
$$
$$
Y+\mathbb{R}\vec{v}+[0,1]\times[0,2]\times[0,1],
$$
where $X$ and $Y$ are points of $\mathbb{R}^3$ $\vec{u}$ and $\vec{v}$ are in the $2$-plane with normal vector $(1,\alpha,\beta)$.
There are thus two reals $\lambda$ and $\mu$ such that
$$
x'-z'=\lambda\vec{u}+h_2
\qquad\textrm{and}\qquad
z'-y'=\mu\vec{v}+h_3,
$$
where $x'$, $y'$ and $z'$ denotes the lifts of $x$, $y$ et $z$, $h_2\in[0,1]\times[0,1]\times[0,2]$ and $h_3\in[0,1]\times[0,2]\times[0,1]$.
Whence
$$
x-y=\lambda\vec{u}+\mu\vec{v}+(h_2+h_3).
$$
Since $h_2+h_3\in[0,2]\times[0,3]\times[0,3]$, this shows that any two points in the lift of the tiling are contained in the slice obtained by moving the box $[0,2]\times[0,3]\times[0,3]$ on the $2$-plane with normal vector $(1,\alpha,\beta)$.
This shows that these colored local rules are weak.\\

Actually, we can do the same by considering a $\vec{v}_1$-directed ribbon instead of the $\vec{v}_2$-directed one, getting a box $[0,3]\times[0,2]\times[0,3]$, or instead of the $\vec{v}_3$-directed one, getting a box $[0,3]\times[0,3]\times[0,2]$.
Since the intersection of the slices obtained via these three different boxes is just the slice obtained via the box $[0,2]\times[0,2]\times[0,2]$.
The weak colored local rules thus have thickness $2$.
\end{proof}

\begin{figure}[hbtp]
\centering
\includegraphics[width=0.9\textwidth]{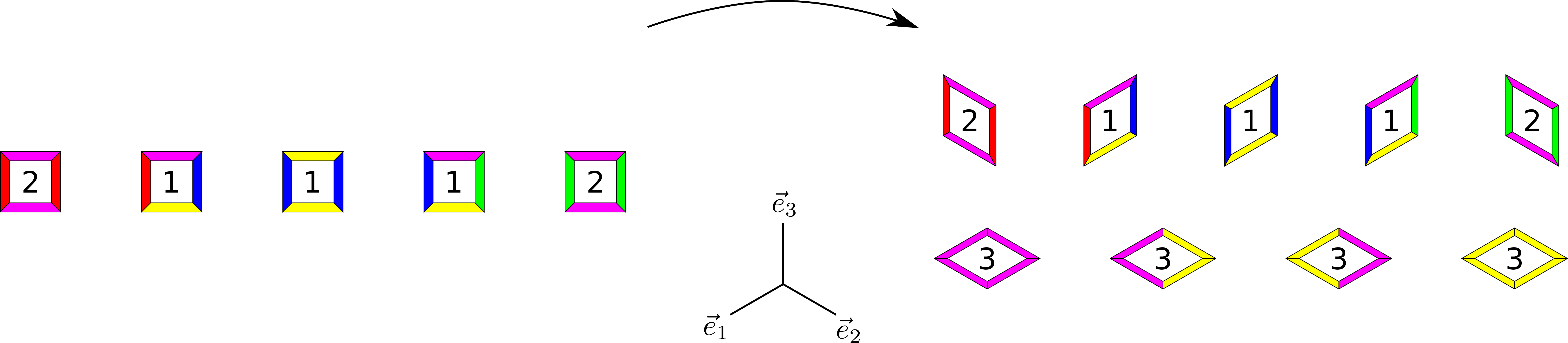}
\caption{A set of Wang tile (left) and the corresponding type $1$ or $2$ ``sheared'' tiles completed with the type $3$ ``transfer'' tiles'' (right).}
\label{fig:shearing_wang}
\end{figure}

\begin{figure}[hbtp]
\centering
\includegraphics[width=0.6\textwidth]{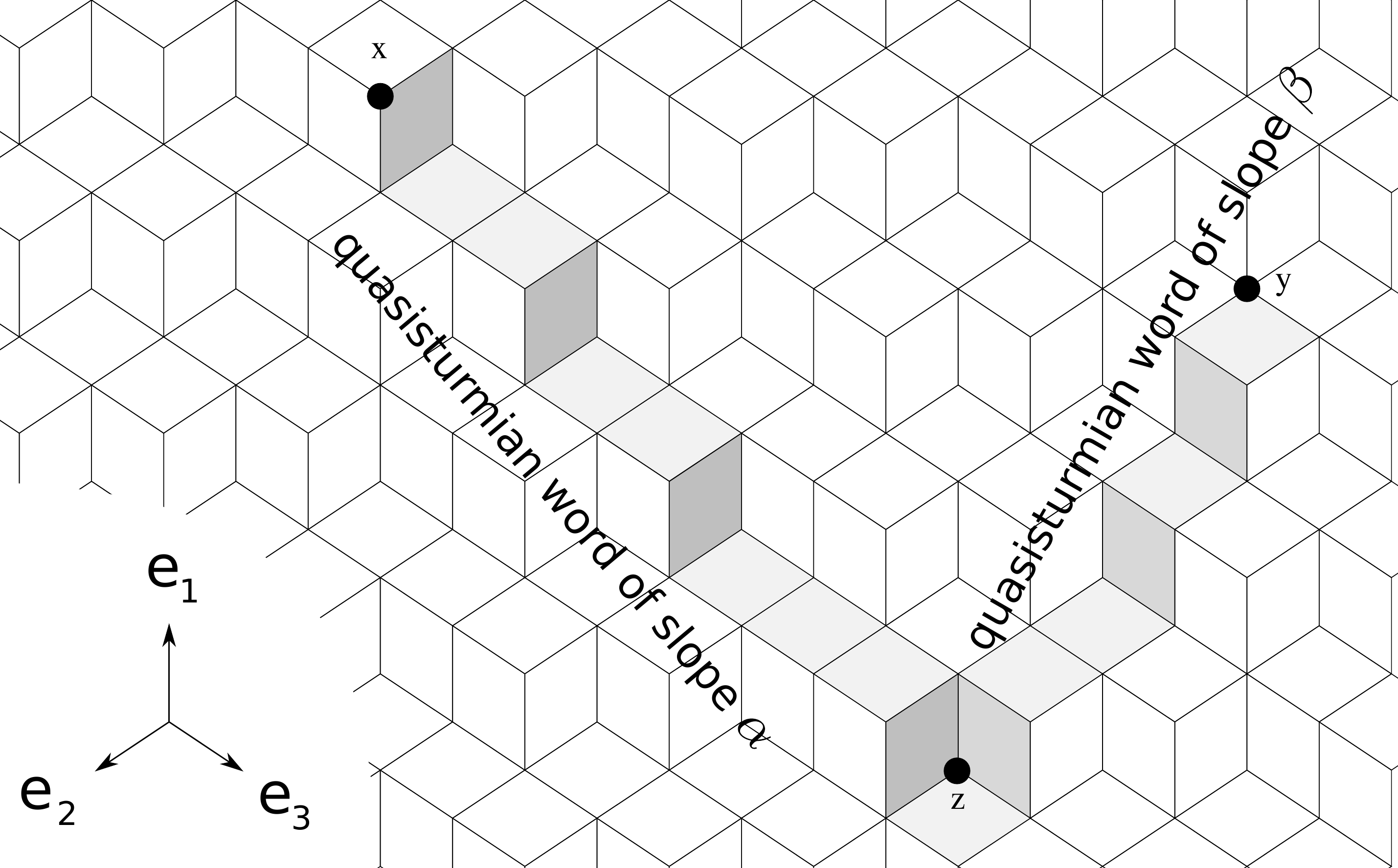}
\caption{The two points $x$ and $y$ are connected by a path made of two ribbons.}
\label{fig:connecting_ribbons}
\end{figure}

Actually, in the proof of the above proposition, we can replace the three subshifts $S'_{\alpha}$, $S'_{\beta}$ and $S'_{\alpha/\beta}$ by the subshift $S'$ formed by all the Sturmian words (with any slope), which is also effective according to Proposition~\ref{prop:sofic_sturmians_computable_slopes}.
This yields that ``planarity is sofic'', formally:

\begin{proposition}\label{prop:weak_local_rules_all_plane}
The set of all the $2$-planes in $\mathbb{R}^3$ has weak colored local rules of thickness $2$.
\end{proposition}

\subsection{The general case of $2$-planes in $\mathbb{R}^3$}

In the general case of a recursively closed set of $2$-planes in $\mathbb{R}^3$, the lines of the three subshifts obtained by projection are no more independent.
We shall synchronize them to prove

\begin{proposition}\label{prop:weak_local_rules_plane_set}
Any recursively closed set of $2$-planes in $\mathbb{R}^3$ has weak colored local rules of thickness $2$.
\end{proposition}

\begin{proof}
Let $A$ be a recursively closed set of slopes.
We assume that the greatest entry of a slope is always the first one (if it is not the case, we split $A$ in sets $A_i$'s with the greatest entry of a vector in $A_i$ being the $i$-th one, and we apply what follows to each $A_i$ to get a set of colored tile $\tau_i$ which produce planar tiling with slope in $A_i$, and we take the union of the $\tau_i$'s).
Hence $A$ can be seen as a subset of $\mathbb{R}^2$, with each $(\alpha,\beta)\in A$ corresponding to a plane of slope $(1,\alpha,\beta)$.\\

As in Prop.~\ref{prop:weak_local_rules_one_plane}, we can find colored tiles such that $\vec{v}_1$-ribbons of any allowed tiling are quasisturmian words of slope $\alpha$ for each first entry $\alpha$ of a vector in $A$ (since the projection on the first entry of $A$ is still a recursively closed set).
Similarly, we enforce the $\vec{v}_2$-ribbons (resp. $\vec{v}_3$-ribbons) to be quasisturmian words of slope $\beta$ (resp. $\alpha/\beta$) for each second entry $\beta$ of a vector in $A$ (resp. for each $\alpha$ and $\beta$ such that there is $(\gamma,\delta)\in A$ with $\alpha/\beta=\gamma/\delta$).
But this is not sufficient because the obtained tile set shall form all the planar tilings of thickness $2$ with slopes in
$$
\{(\alpha,\beta)~|~\exists (\gamma,\delta)\in A,~(\alpha,\delta)\in A \textrm{ or } (\gamma,\beta)\in A \textrm{ or } \alpha/\beta=\gamma/\delta\},
$$
which is in general not equal to $A$ (except if $A$ is the intersection of a pro\-duct of two intervals by some lines).
We thus also need to {\em synchronize} the ribbons.\\

We first modify the tile set so that, given a possible tiling whose $\vec{v}_1$-ribbons are quasisturmian words of slope $\alpha$, then for any $\beta$ such that $(\alpha,\beta)\in A$, a Sturmian word of slope $\beta$ shall be ``hidden'' in each of these $\vec{v}_1$-ribbons (with the same $\beta$ for all the $\vec{v}_1$-ribbons of a given tiling), namely on the tiles with edges $\vec{v}_1$ and $\vec{v}_2$.
Let us explain how to ``hide'' these Sturmian words.
We introduce the one-dimensional subshift
\begin{equation}\label{eq:S_A2}
\tilde{S}_A=\left\{
x\in\{0,1,\tilde{1}\}^{\mathbb{Z}}~:~
\exists
\begin{array}{l}
(\alpha,\beta)\in A,\\
\rho,\tau\in[0,1]
\end{array}
,~
\begin{array}{l}
\varphi(x)=s_{\alpha,\rho}\\
\psi(x)=s_{\beta,\tau}
\end{array}
\right\}
\end{equation}
where $\varphi$ and $\psi$ are the morphisms over words defined by
$$
\varphi~:~\left\{\begin{array}{ccc}
0 &\mapsto& 0\\
1 &\mapsto& 1\\
\tilde{1} &\mapsto& 1
\end{array}\right.
\qquad\textrm{and}\qquad
\psi~:~\left\{\begin{array}{ccc}
0 &\mapsto& \varepsilon\\
1 &\mapsto& 0\\
\tilde{1} &\mapsto& 1
\end{array}\right.,
$$
and $\varepsilon$ denotes the empty word.
In other words, $\psi$ reveals the ``hidden'' Sturmian word $s_{\beta,\tau}$ which is encoded in the distinction between $1$ and $\tilde{1}$, while $\varphi$ removes this distinction by identifying $1$ and $\tilde{1}$.
This subshift is effective.
Indeed, we can enumerate all the words over $\{0,1,\tilde{1}\}$ by lexicographic order and, given such a word $w$, compute $\varphi(x)$ and $\psi(w)$, and check that their are factors of Sturmian words of slope $\alpha$ and $\beta$ for some $(\alpha,\beta)\in A$ (this is possible because $A$ is recursively closed).
We can then proceed as we did in Section~\ref{sec:quasisturmian}.
We first extend $\tilde{S}_A$ to a two-dimensional subshift with equal lines which is sofic according to Th.~\ref{th:subaction}.
We then relax it by allowing lines to be quasisturmian words of slope $\alpha$ - with the Sturmian words of slope $\beta$ remaining hidden.
We finally transform it into a finite tile set as in Prop.~\ref{prop:weak_local_rules_one_plane} - with the Sturmian words of slope $\beta$ being written on the tiles with edges $\vec{v}_1$ and $\vec{v}_2$.\\

We proceed similarly on the $\vec{v}_2$-ribbons, but the $\vec{v}_2$-ribbons are now quasisturmian words of slope $\delta$ such that there is $\gamma$ with $(\gamma,\delta)\in A$, and we enforce the slope of the hidden Sturmian word (also written on the tiles with edges $\vec{v}_1$ and $\vec{v}_2$) to be {\em the same} as the one of the quasisturmian word it is hidden in.\\

Hence, for any valid tiling, there are now $(\alpha,\beta)$ and $(\gamma,\delta)$ in $A$ such that the $\vec{v}_1$-ribbons are quasisturmians words of slope $\alpha$ with a hidden Sturmian word of slope $\beta$, while the $\vec{v}_2$-ribbons are quasisturmians words of slope $\delta$ with a hidden Sturmian word of slope $\delta$.\\

The last step is to ``connect'' the tiles with edges $\vec{v}_1$ and $\vec{v}_2$ in order to enforce the equality of the Sturmian words which are hidden in.
This would yields $\beta=\delta$, hence quasisturmian $\vec{v}_2$-ribbons of slope $\beta$, that is, a slope $(\alpha,\beta)\in A$ for the tiling.
To do this, we allow the letters written on the tiles with edges $\vec{v}_1$ and $\vec{v}_2$ (there are only two possible letters) to flow as follows through the tiles (Fig.~\ref{fig:synchronization2}):
\begin{itemize}
\item the tiles with edges $\vec{v}_1$ and $\vec{v}_3$ transmit each letter between its two $\vec{v}_1$-edges;
\item the tiles with edges $\vec{v}_2$ and $\vec{v}_3$ transmit each letter between its two $\vec{v}_2$-edges;
\item the tiles with edges $\vec{v}_1$ and $\vec{v}_2$ transmit each letter between one of its $\vec{v}_1$-edge and the $\vec{v}_2$-edge which starts at the same point (no matter which edge is chosen, but we shall do the same choice uniformly for every such tile), while only the letter that is written on the tile is allowed to appear on the two remaining edges.
\end{itemize}
This way, the hidden letters will flow along non-intersecting ``diagonal curves'' which cross both $\vec{v}_1$- and $\vec{v}_2$-ribbons, enforcing letter by letter the hidden Sturmian words on the $\vec{v}_1$-ribbons to be equal to the one on the $\vec{v}_2$-ribbons (Fig.~\ref{fig:synchronization3}).
\end{proof}

\begin{figure}[hbtp]
\centering
\includegraphics[width=\textwidth]{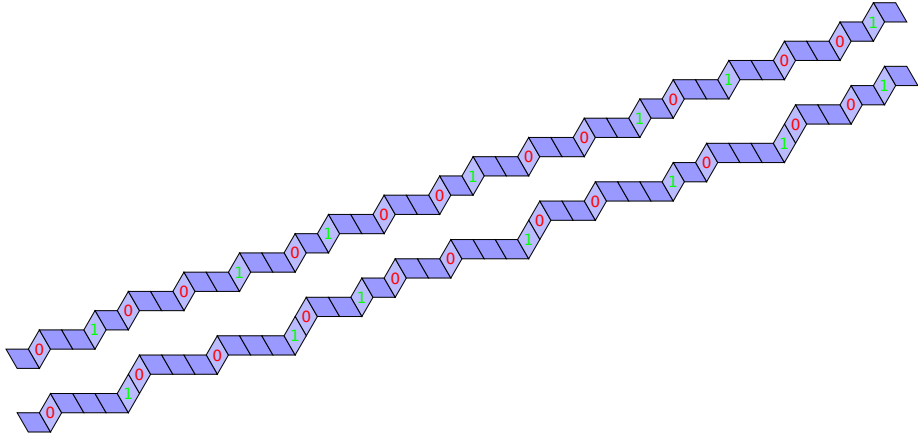}
\caption{
  A Sturmian (top) and a quasisturmian (bottom) $\vec{v}_1$-ribbons, both with the same hidden quasisturmian word (letters written on the tiles).
}
\label{fig:synchronization1}
\end{figure}

\begin{figure}[hbtp]
\centering
\includegraphics[width=0.5\textwidth]{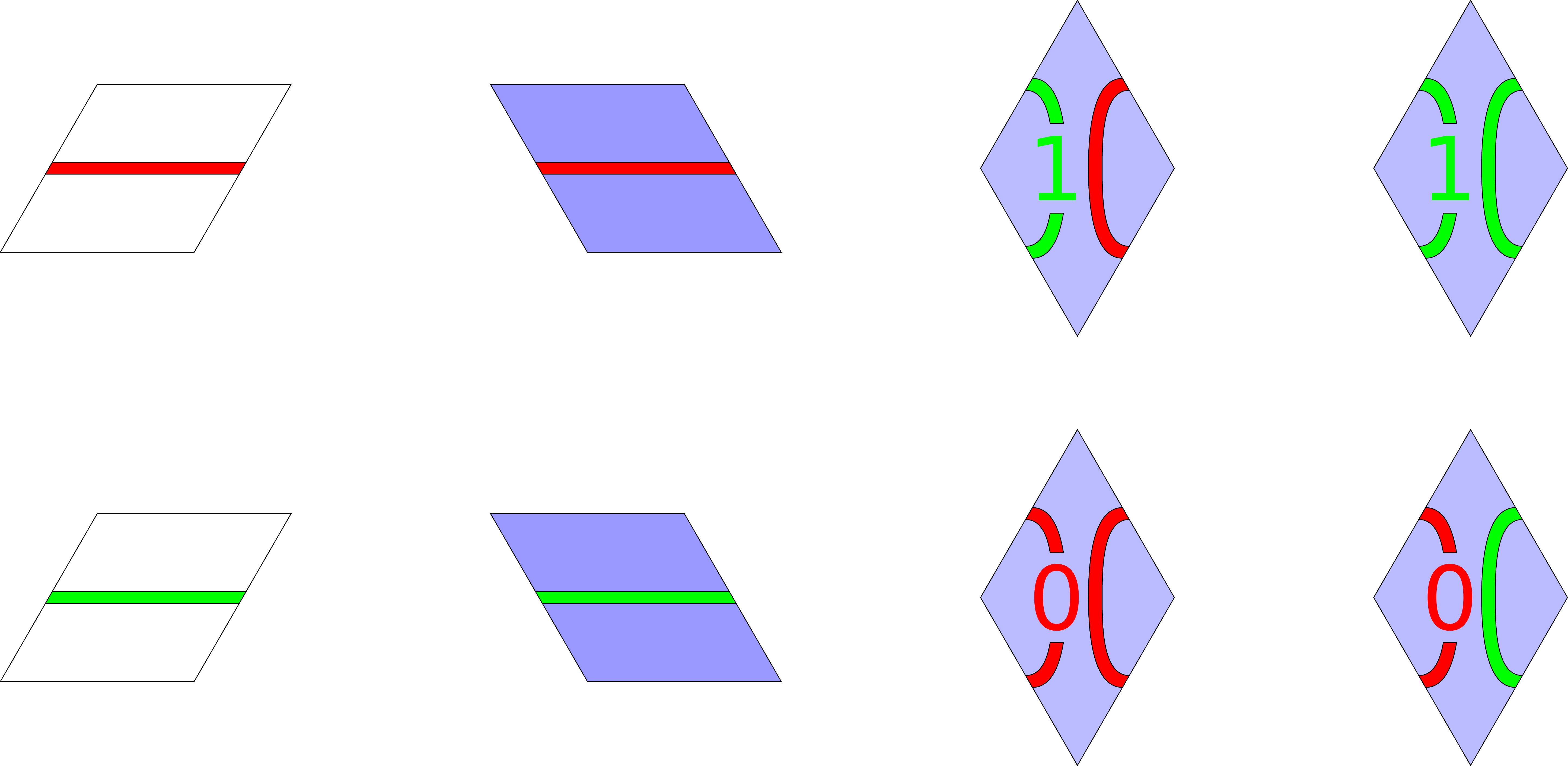}
\caption{
  The tiles with the decorations that transmit the letters of the hidden Sturmian words to synchronize the $\vec{v}_1$- and $\vec{v}_2$-ribbons.
}
\label{fig:synchronization2}
\end{figure}

\begin{figure}[hbtp]
\centering
\includegraphics[width=\textwidth]{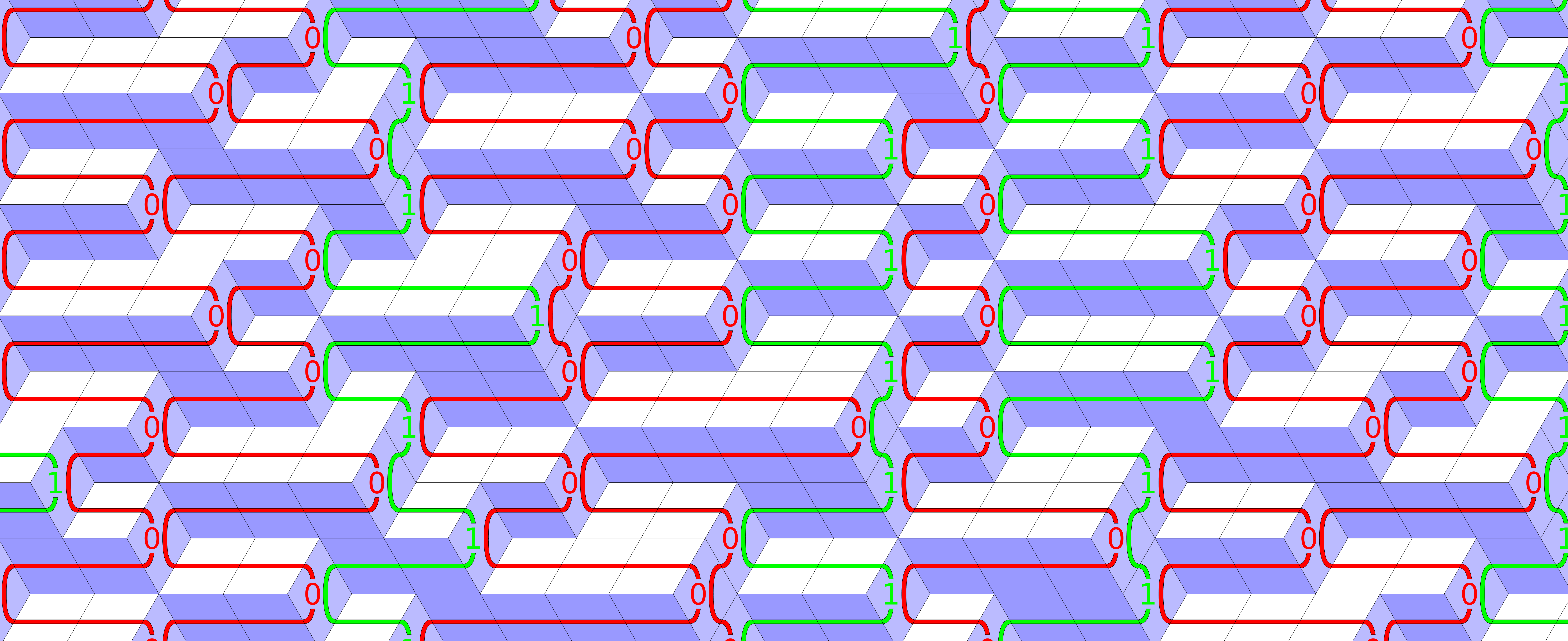}
\caption{
  A tiling whose $\vec{v}_1$-ribbons are quasisturmian of slope $\alpha$ with a hidden Sturmian of slope $\beta$, while the $\vec{v}_2$-ribbons are quasisturmian of slope $\gamma$ with a hidden Sturmian of slope $\gamma$ ($\vec{v}_1$- and $\vec{v}_2$-ribbons are those which contain the tiles with a hidden letter).
  The tile decorations transfer the hidden letters to ensure that the hidden Sturmian words on both ribbons are equal, that is, $\gamma=\beta$.
}
\label{fig:synchronization3}
\end{figure}

\subsection{Higher codimension and dimension}

The last step to prove Theorem~\ref{th:main} is to extend Prop.~\ref{prop:weak_local_rules_plane_set} to higher dimension and codimension tilings.
This is a bit technical but rather simple.\\

For higher codimensions, we proceed by induction.
Our induction hypothesis is that any computable planar $n\to 2$ tiling admits weak local rules.
This holds for $n=3$ according to the previous section.
Let now $\mathcal{T}$ be an computable planar $(n+1)\to 2$ tiling.
For any basis vector $\vec{e}_i$, we project the lift of $\mathcal{T}$ along $\vec{e}_i$ to get the lift of an computable planar $n\to 2$ tiling, say $\mathcal{T}_i$.
By assumption, $\mathcal{T}_i$ admits local rules: let $\tau_i$ be a tile set whose tilings are at distance at most $w$ from $\mathcal{T}_i$.
We complete $\tau_i$ by adding the tiles with a $\vec{v}_i$-edge (that is, the tiles which disappeared from $\mathcal{T}$ by projecting along $\vec{e}_i$), with each of these tiles having no decoration on its $\vec{v}_i$-edges, and on the other edges a unique decoration that could be any of those appearing on an edge of a tile in $\tau_i$.
These new tiles thus just transfer decorations between the tiles of $\mathcal{T}_i$ (see Fig.~\ref{fig:higher_codim} for $n+1=4$).
Last, we define the tile set $\tau$ as the cartesian product of all the $\tau_i$'s (as we did for $\tau_\alpha$ and $\tau_\beta$ in the previous section).
This allows only $n+1\to 2$ tilings at distance at most $w'$ from $\mathcal{T}$ - in particular $\mathcal{T}$ itself.
This shows that $\mathcal{T}$ admits local rules.\\

\begin{figure}[hbtp]
\centering
\includegraphics[width=\textwidth]{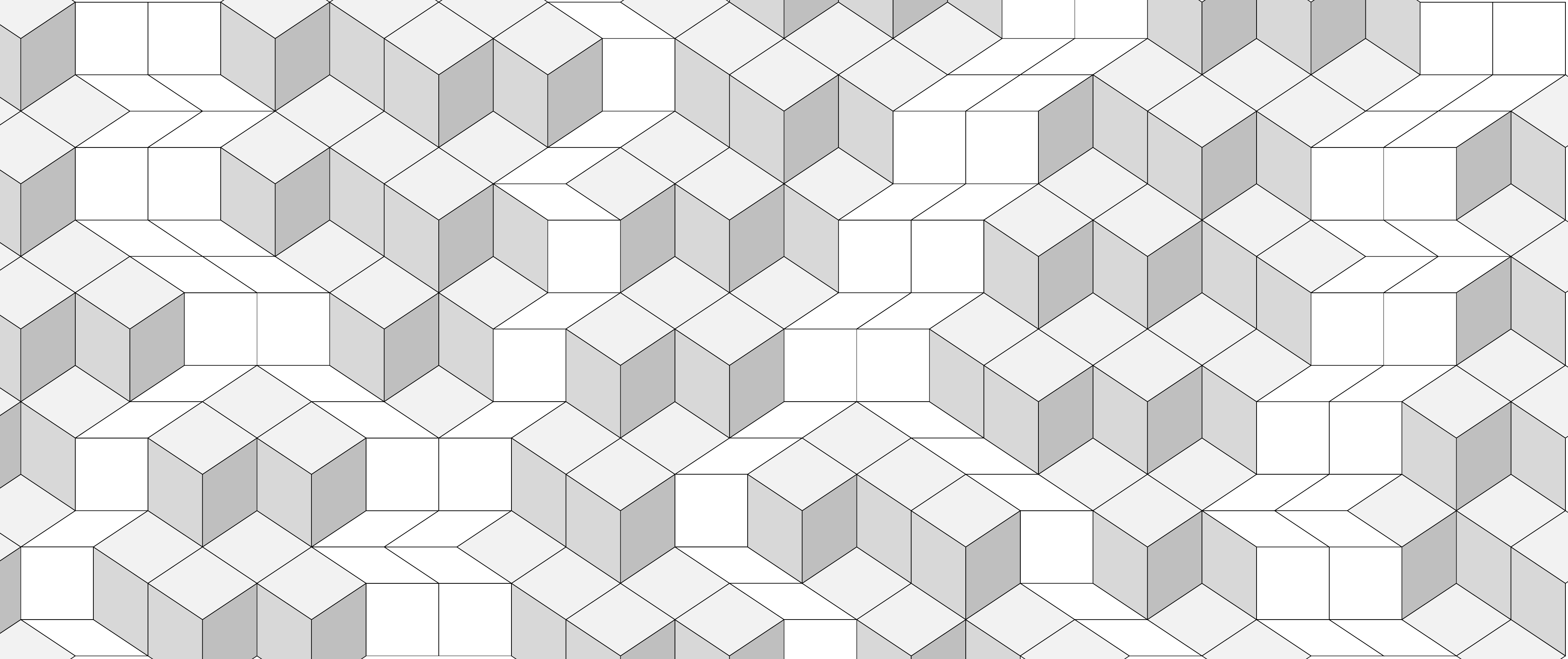}
\caption{
A $4\to 2$ tiling by shaded $\tau_i$ tiles and white additional tiles, with $\vec{v}_i$ being here horizontal.
The white tiles simply transfer horizontally the decorations of the shaded tiles.
By contracting each $\vec{v}_i$-edges to a point, the white ribbons disappear: we get a $3\to 2$ tiling by $\tau_i$, at distance at most $w'$ from $\mathcal{T}_i$.}
\label{fig:higher_codim}
\end{figure}

For higher dimensions, we also proceed by induction.
Our induction hypothesis is, for a fixed $n$, that any computable planar $n\to d$ tiling, $d<n$, admits local rules.
This holds for $d=2$ according to the above paragraph.
Let now $\mathcal{T}$ be an computable planar $n\to (d+1)$ tiling, with $d+1<n$.
Fix $i\in\{1,\ldots,n\}$.
For two tiles $T$ and $T'$ of $\mathcal{T}$, write $T\sim T'$ if these tiles share a $\vec{v}_i$-edge and let $\simeq$ be the transitive closure of the relation $\sim$.
Denote by $(\mathcal{T}_k)_{k\in\mathbb{Z}}$ the equivalence classes of $\simeq$, such that, for any $k$, $\mathcal{T}_k$ and $\mathcal{T}_{k+1}$ can be connected by a path which does not cross any other equivalence class (the $\mathcal{T}_k$'s play the role of $\vec{v}_i$-ribbons in the previous section).
By contracting all the $\vec{v}_i$-edges of a $\mathcal{T}_k$ ({\em flattening}), one gets a planar $n\to d$ tiling.
Its slope moreover depends only on the slope of $\mathcal{T}$, and in particular it is effective.
This allows to see $\mathcal{T}$ as a sequence of ``stacked'' parallel computable planar $n\to d$ tilings (namely the flattened $\mathcal{T}_k$'s), with the remaining tiles containing no $\vec{v}_i$-edge.
By induction, there exists a finite tile set $\tau_i$ whose tilings are at bounded distance $w$ from any of the flattened $\mathcal{T}_k$'s (since they are all parallel).
It is straighforward to ``unflatten'' $\tau_i$ to get a tile set $\tilde{\tau}_i$ whose tilings are at bounded distance $w$ from any of the $\mathcal{T}_k$'s.
We complete $\tilde{\tau}_i$ by adding the tiles without $\vec{v}_i$-edge (that is, the tiles lying between the stacked $\mathcal{T}_k$'s), with decorations being just transferred between consecutive $\mathcal{T}_k$'s along the direction $\vec{v}_i$ (as done in the previous section to transfer decorations between consecutive $\vec{v}_i$-ribbons).
The last step is (as in the previous section again) to define the cartesian product $\tau$ of the tile sets $\tilde{\tau}_i$, $i=1,\ldots,d$: its tilings are those at bounded distance $w$ from $\mathcal{T}$ - in particular $\mathcal{T}$ itself.
This shows that $\mathcal{T}$ admits local rules.

%%%%%%%%%%%%%%%%%%%%%%%%%%%%%%%%%%%%%%%%%
%%%%%%%%%%%%%%%%%%%%%%%%%%%%%%%%%%%%%%%%%
%%%%%%%%%%%%%%%%%%%%%%%%%%%%%%%%%%%%%%%%%
\section{Removing colors}
\label{sec:uncolor}

In this last section, we prove that colored weak local rules can be replaced by uncolored weakened local rules with almost no loss on the set of planes that can be enforced (Prop.~\ref{prop:color_removing}, below).
This shows that the power of weakened and weak local rules are dramatically different (uncolored weak local rules can indeed only enforce algebraic planes, recall Sec.~\ref{sec:results}), althoug the only additional property of the latter is that at least one perfectly planar tiling (that is, of thickness $1$) can be formed (recall Sec.~\ref{sec:local_rules}).
The proof is technical but rather simple: we use the allowed fluctuations around the plane to encode the colors of the local rules.
This is where the small loss of power comes: we need {\em non-degenerated} planes to encode suitably colors in fluctuations.\\

A $d$-plane $E\subset\mathbb{R}^n$ is said to be {\em degenerated} if it is the slope of a $n\to d$ planar tiling of thickness one which does not contain all the $\binom{n}{d}$ possible different tiles (equivalently, $E$ has at least one Grassmann coordinate equal to zero).
If $E$ is non-degenerated, then it contains a vertex which belongs to exactly $d+1$ tiles, and the region tiled by these $d+1$ tiles can be tiled in exactly one another way by the same tile.
The operation which exchange these two possible tilings is called a {\em flip}: it consists in translating each of the $d+1$ tiles along the vectors shared by the $d$ other ones (see Fig.~\ref{fig:flip}).
Such a flip yields a new $n\to d$ planar tiling of slope $E$, whose thickness can however increase by one.\\
\begin{figure}[hbtp]
  \centering
\includegraphics[width=0.9\textwidth]{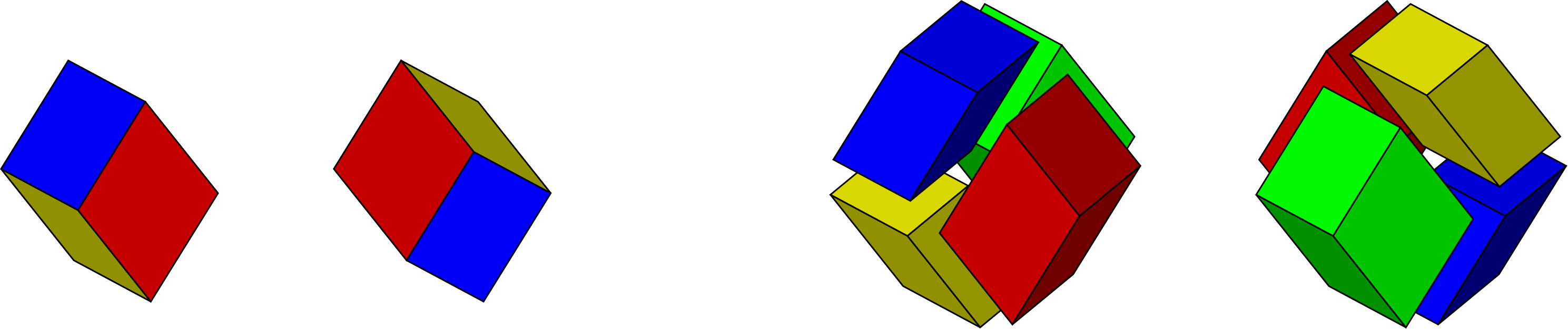}
\caption{Flips in $\mathbb{R}^2$ (left) and $\mathbb{R}^3$ (right, exploded view)}
\label{fig:flip}
\end{figure}

\noindent We shall use flips to remove colors:

\begin{proposition}\label{prop:color_removing}
A set of non-degenerated planes enforced by colored weak local rules of thickness $t$ can also be enforced by weakened local rules of thickness $t+1$.
\end{proposition}

\begin{proof}
For the sake of simplicity, we consider a single $2$-plane $E\subset\mathbb{R}^n$.
The general case is similar.
Consider a tile set with colored boundaries (as for Wang tiles) that can form all the planar tilings with slope $E$ and thickness at most $t$.
Among these tilings, let $\mathcal{P}$ be one with thickness $1$ (it is thus repetitive).\\

Let us fix two types of ribbons and, for each type, mark one of $k$ consecutive ribbons of this type in $\mathcal{P}$ (the parameter $k$ shall be later chosen).
By repetitivity of $\mathcal{P}$, there is a uniform bound on the distance between two consecutive intersections of a ribbon with the ribbons of the other type.
The ribbons thus draw on $\mathcal{P}$ a sort of grid whose cells have a perimeter proportional to $k$ and an area proportional to $k^2$ (see Fig.~\ref{fig:remove_colors_a}).\\

\begin{figure}[hbtp]
\includegraphics[width=\textwidth]{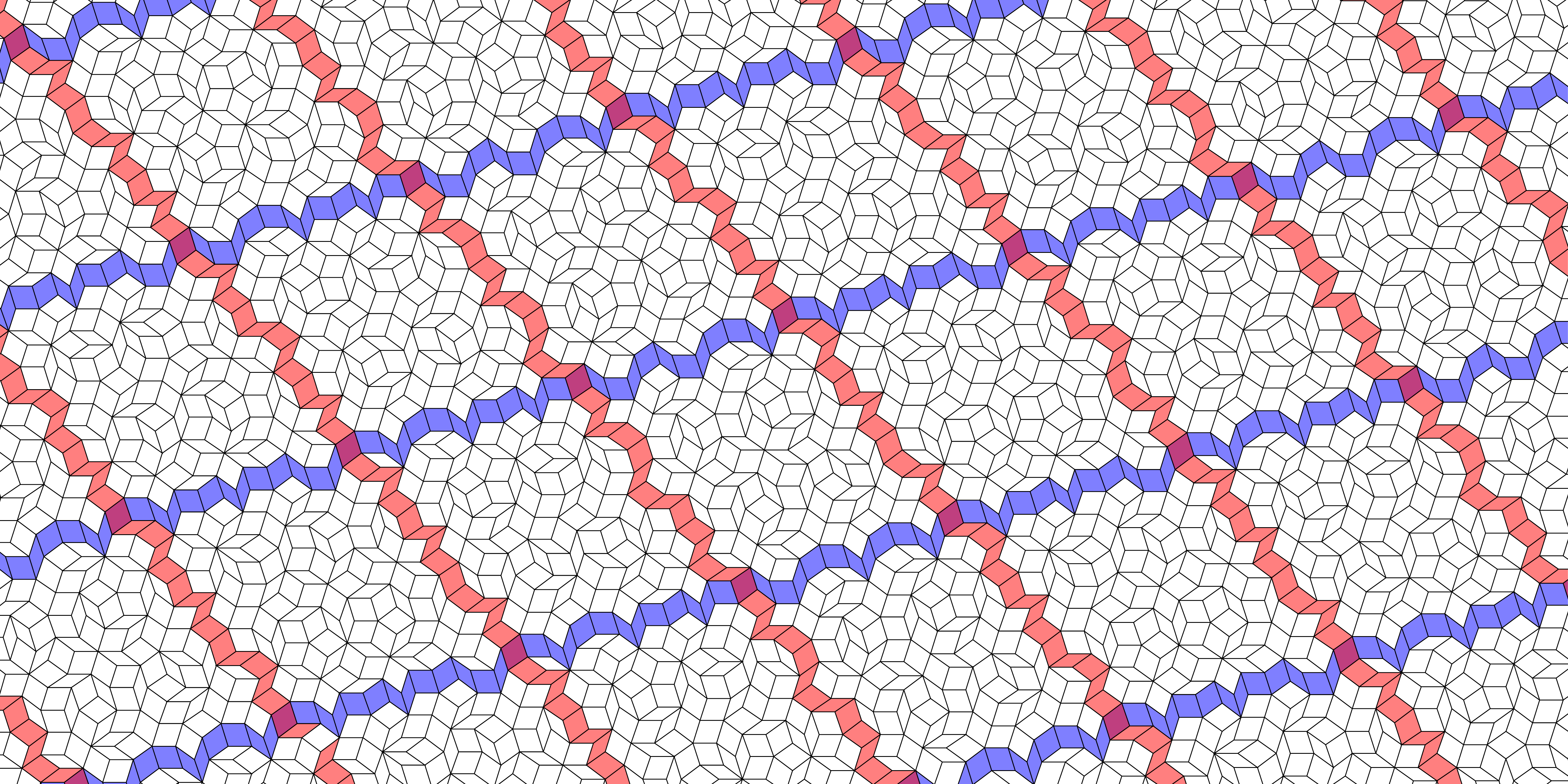}
\caption{
A grid on the tiling $\mathcal{P}$ obtained by marking one over three ribbons of two fixed types (colors of tiles are not depicted).
This define meta-tiles, the boundary colors of which can be encoded by flips performed in their interior.
}
\label{fig:remove_colors_a}
\end{figure}

If one sees these cells as meta-tiles, we get a set of colored meta-tiles which can only form a subset of the tilings formed by the initial tile set.
This subset is non-empty since it contains at least $\mathcal{P}$.
Moreover, there are at most $(c+n)^p$ different meta-tiles, where $c$ is the number of colors used by the initial tiles, $p$ is the maximal perimeter of the cells and $n$ is the number of possible directions for the tile edges.\\

Now, since the plane is non-degenerated, $\mathcal{P}$ contains at least one flip.
By repetitivity of $\mathcal{P}$, any pattern -- in particular a grid cell -- contains a number $f$ of flips which is proportional to its surface.
These $f$ flips can encode $2^f$ numbers (by performing or not each of them).
Since the area of a meta-tile grows with $k$ as the square of its perimeter, for $k$ big enough one has $2^f>(c+n)^p$, that is, there are enough flips to encode in each meta-tile the colors of the tile on its boundary.
Let then $\mathcal{P}'$ be the tiling obtained by performing the flip to encode meta-tiles boundaries and then removing the colors.\\

Let $F$ be the set of patterns which do not appear in $\mathcal{P}'$ and whose diameter (that is, the maximal number of tiles crossed by a line segment joining two of its vertices) is twice the maximal diameter of the meta-tiles (that is, the grid cells).
Consider a tiling without pattern in $F$ (such a tiling exists, for example $\mathcal{P}'$ itself).
On the one hand, any tile belongs to a meta-tile (whose boundary colors are encoded by flips).
On the other hand, whenever two meta-tiles are adjacent, their flips encode boundary colors which match.
However, nothing yet ensures that meta-tiles do not overlap.\\

To fix that, we shall use additional flips to put a special marking on each meta-tile that shall no be confused with the encoding of boundary colors.
The existence of such additional flips is not problematic: it suffices to increase the size of the meta-tiles.
There is then several way to proceed.
For example, consider two particular flips and let $\vec{x}$ denotes the vector which goes from one to the other.
By repetitivity, this pair of flips appears in each meta-tile for big enough meta-tiles.
Let us perform such a pair of flips in each meta-tile.
It then suffices, to avoid meta-tile overlap, to forbid the use of a flip at position $\vec{y}$ to encode boundary colors if there is another flip at position $\vec{y}+\vec{x}$.\\

Now, any tiling without pattern in $F$ is a tiling by meta-tiles.
By replacing the flips by the colors that they encode, this yields a tiling by the initial tile set, that is, a planar tiling of slope $E$ and thickness $t$.
These coding flips do not modify the slope, but they can increase the thickness by one (no more since these flips are disjoint).
Therefore the maximal thickness is $t+1$, and the local rules are weakened (instead of weak) because the tilings could all have thickess at least two.
This proves the claimed result.
\end{proof}

\noindent With the above proposition, Theorem~\ref{th:main} immediatly yields:

\begin{corollary}\label{cor:uncolor}
A set of non-degenerated planes is enforced by weakened local rules iff it is recursively closed.
\end{corollary}

%%%%%%%%%%%%%%%%%%%%%%%%%%%%%%%%%%%%%%%%%
%%%%%%%%%%%%%%%%%%%%%%%%%%%%%%%%%%%%%%%%%
%%%%%%%%%%%%%%%%%%%%%%%%%%%%%%%%%%%%%%%%%
\appendix
\section{Proof of Proposition~\ref{prop:sturmian_lines_non_sofic}}
\label{sec:appendix}

We shall here prove that if $A\not\subset \mathbb{Q}$, then the following subshift is not sofic
$$
S''_A=\{x\in\{0,1\}^{\mathbb{Z}^2}~:~\exists \alpha\in A,~\forall m\in\mathbb{Z},~\exists \rho\in[0,1],~x(m,\cdotp)=s_{\alpha,\rho}\}.
$$
The key point is that its number of $n\times n$ patterns grows faster than an exponential in $n$, but slower than an exponential in $n^2$.
Let us first prove a lemma:

\begin{lemma}\label{lem:config_complexe}
  If $A\not\subset \mathbb{Q}$, then for any $n>0$, there exist $d>0$ and $x\in S''_A$ such that any $d\times d$ pattern of $x$ contains at least $n^n$ different disjoint $n\times n$ patterns.
\end{lemma}

\begin{proof}
  Let $n>0$.
  Let $\alpha\in A\backslash\mathbb{Q}$ and consider a biinfinite Sturmian word $w$ of slope $\alpha$.
  It has $n+1$ different subwords of length $n$.
  Moreover, it is {\em uniformly recurrent}, that is, there is $c>0$ such that any subword of length $b$ of $w$ contains all the subwords of length $n$.
  Consider, {\em e.g.}, the subword of length $b$ at position $0$ and let $0\leq i_0,\ldots,i_n<b$ be such that the subwords of length $k$ at these positions are different.
  We shall construct $x$ by stacking shifted copies of $w$.\\
  
  For $0\leq p<n^n$, let $B_p$ be the stripe of $n$ stacked lines defined as follows: its $k$-th line ($0\leq k<n$) is $w$ shifted by $i_k$ letters to the left, where $i_k$ is the $k$-th digit of the $n$-ary expansion of $p$.
  Let $\mathcal{P}_p$ denotes the $n\times n$ patterns at position $0$ in $B_p$.
  This definition ensures that the $\mathcal{P}_p$'s are different for each $p$. 
  Moreover, $\mathcal{P}_p$ reoccurs in $B_p$ at each position where the subword $u$ of length $b+n$ at position $0$ in $w$ reoccurs (because $w$ is never shifted by more that $b$).
  Since $u$ is itself uniformly recurrent, there is $c>0$ such that $\mathcal{P}_p$ appears in any rectangle of $B_p$ of width $c$ and height $n$ (the height of $B_p$).\\

  Then, we stack the stripes $B_p$'s one above the other to get a large stripe made of $n^{n+1}$ stacked shifted $w$.
  There is disjoint occurrences of each of the $\mathcal{P}_p$'s in any rectangle of $B_p$ of width $c$ and height $n^{n+1}$ (the height of this large stripe).
  Last, we stack periodically this large stripe to get an element $x\in S''_A$.
  There is disjoint occurrences of each of the $\mathcal{P}_p$'s in any rectangle of width $c$ and height $n^{n+1}+n$ (we add the height of a $B_p$ to the height of the large stripe in order to ensure that the rectangle contains all the stacked $B_p$'s even if its basis falls inside a $B_p$).
  With $d:=\max(c,n^{n+1}+n)$ this yields the desired $x\in S''_A$.
\end{proof}

\noindent We are now in a position to prove Proposition~\ref{prop:sturmian_lines_non_sofic}:\\

\begin{proof}
  Assume that $S''_A$ is sofic and let us get a contradiction on its {\em entropy}
  $$
  H''_A:=\limsup_{n\infty}\frac{\ln(\textrm{number of $n\times n$ patterns})}{n^2}.
  $$
  First, there is $c>0$ such that the number of Sturmian words of length $n$ with any slope is bounded by $cn^3$ (see, {\em e.g.}, \cite{Mignosi-1991}).
  The entropy of $S''_A$ is thus zero:
  $$
  H''_A\leq \limsup_{n\infty}\frac{n\ln(cn^3)}{n^2}=0.
  $$
  
  Then, since $S''_A$ is sofic, there is a finite type subshift $S'''_A$ and a factor map $\pi$ such that $S''_A=\pi(S'''_A)$.
  Let $b$ be the size of the alphabet $S'''_A$ is defined on and let $k\geq 0$ be an upper bound on the diameter (for, {\em e.g.}, the infinite norm) of the forbidden patterns which define $S'''_A$.
  Let us call {\em corona} of a $n\times n$ pattern which appears in an element of $S'''_A$ the pattern formed by the letters at distance at most $k$ from this $n\times n$ pattern, and not inside it.
  There is thus at most $b^{(n+2k)^2-n^2}$ different coronas in $S'''_A$.
  Since $n^n=\textrm{e}^{n\ln n}$ grows faster than $b^{(n+2k)^2-n^2}=b^{O(n)}$, this ensures with Lemma~\ref{lem:config_complexe} that there exist $n>0$, $d>0$ and $x\in S''_A$ such that in any $d\times d$ pattern of $x$, at least two disjoint $n\times n$ patterns are surrounded by the same corona in the preimage of $x$ under $\pi$.\\
  
  Now, for $m>0$, consider an $md\times md$ pattern of the preimage of $x$ under $\pi$.
  Decompose it in a $m\times m$ square grid of $d\times d$ patterns.
  Each of these $d\times d$ pattern contains two disjoint $n\times n$ patterns with the same corona but whose image under $\pi$ are different.
  Since they have the same corona and are disjoint, we can swap them and stay in $S'''_A$.
  But since their image under $\pi$ are different, and since we can do this independently on each of the $m^2$ patterns of size $d\times d$, there are at least $2^{m^2}$ different $md\times md$ patterns in $S''_A$.
  This yields a positive lower bound on the entropy:
  $$
  H''_A\geq \limsup_{m\infty}\frac{\ln(2^{m^2})}{(md)^2}=\frac{\ln 2}{d^2}>0.
  $$
  With $H''_A=0$, this yields the desired contradiction.
\end{proof}

%\bibliographystyle{alpha}
%\addcontentsline{toc}{section}{References}
%\bibliography{weak_effective}

%pour ArXiv :

\end{document}